\documentclass[11pt]{amsart}
\usepackage{amscd}
\usepackage{amsfonts}
\usepackage{amssymb}
\usepackage{latexsym}
\usepackage[arrow, matrix, curve]{xy}

\usepackage{graphics}
\usepackage{color}

\newcommand{\ncm}{\newcommand}

\newtheorem{theorem}{Theorem}[section]
\newtheorem{prop}[theorem]{Proposition}
\newtheorem{lemma}[theorem]{Lemma}
\newtheorem{cor}[theorem]{Corollary}
\newtheorem{lem&def}[theorem]{Lemma \& Definition}
\newtheorem{definition}[theorem]{Definition}
\newtheorem{example}[theorem]{Example}
\newtheorem{remark}[theorem]{Remark}

\def\C{\mathbb{C}\,}
\def\Z{\mathbb{Z}\,}

\def\R{\mathbb{R}\,}

\ncm{\Res}{\mbox{\rm Res}}

\def\End{\mbox{\rm End}\,}

\def\Hom{\mbox{\rm Hom}\,}
\def\ot{{\otimes}}    %tensor product
\def\|{\, | \, }
\def\bra{\langle}
\def\ket{\rangle}

\ncm{\rarr}[1]{\stackrel{#1}{\longrightarrow}}
\ncm{\larr}[1]{\stackrel{#1}{\longleftarrow}}

\def\eps{\varepsilon}

\def\du1{\hat 1}
\def\bra{\langle}
\def\ket{\rangle}

\def\-1{_{(-1)}}
\def\0{_{(0)}}
\def\1{_{(1)}}
\def\2{_{(2)}}
\def\3{_{(3)}}

\def\du1{\hat 1}

\begin{document}

\title[Induction-restriction depth of subgroups]{On subgroup depth}
\author{Sebastian  Burciu}
\address{Inst.\ of Math.\ ``Simion Stoilow'' \\
Romanian Academy, P.O. Box 1-764 \\
 RO-014700, Bucharest, Romania}
\email{smburciu@syr.edu}
\author{Lars Kadison}
\address{Department of Mathematics \\
University of Pennsylvania \\
D.R.L.\ 209 S.\ 33rd St. \\ Philadelphia, PA 19104-6395 \\ Current
address: Departamento de Matematica \\ Faculdade de Ciencias da
Universidade do Porto \\Rua Campo Alegre, 687\\ 4169-007 Porto, 
Portugal}
\email{lkadison@math.upenn.edu}
%\URL{www.math.upenn.edu/~lkadison}

\author{Burkhard K\"ulshammer}
\address{Mathematisches Institut\\
Friedrich-Schiller-Universit\"at \\
 07737 Jena \\
Germany}
\email{kuelshammer@uni-jena.de}
\thanks{}
\date{\today}

\begin{abstract}
We define a notion of depth for an inclusion of multimatrix algebras
$B \subseteq A$ based on a comparison of powers of the
induction-restriction table $M$ (and its transpose matrix). This
notion of depth coincides with the depth from \cite{K1}. In
particular depth $2$ extensions coincide with normal extensions as
introduced by Rieffel in \cite{R}. For a group extension $H \subset
G$ a necessary depth $n$ condition is given in terms of the core of
$H$ in $G$. We prove that the subgroup depth of symmetric groups
 $S_n < S_{n+1}$ is $2n-1$. An appendix
by S.~Danz and B.~K\"ulshammer determines the subgroup depth of
alternating groups $A_n < A_{n+1}$ and of some other group
extensions.
\end{abstract}

\maketitle

\section{Introduction}

Depth two is an algebraic notion for noncommutative ring extensions with a Galois theory associated to it
\cite{KN, KS}.
If applied to a subalgebra pair of quantum algebras, depth two is a notion of normality that extends ordinary  normality for subgroups and Hopf subalgebras \cite{K2,KK,BK2}.  A Hopf subalgebra $K$ is
normal in a finite dimensional Hopf algebra $H$ if and only if $H$ is a depth two ring extension of $K$:
 see \cite{BKu} for this theorem and how it generalizes to faithfully flat, finitely generated
 projective Hopf algebra extensions over a commutative
 base ring as well as one-sided versions of normality,
depth two and Hopf-Galois extension.

 For $M$ the matrix of the induction-restriction table
for a subalgebra pair of semisimple $\C$-algebras $B \subseteq A$, the depth two condition is given by a matrix inequality $MM^tM \leq qM$ for some $q \in \Z_+$, an observation in \cite{BK} that we build on in this paper.
Recall that a Hopf subalgebra $K $ is normal in a Hopf algebra $H$ if $HK^+ = K^+ H$ where $K^+$ is the maximal ideal $\ker \eps$
restricted to $K$. A predecessor of this definition is Rieffel's definition of a  normal subring: a semisimple subalgebra $B$ in a semisimple algebra $A$ is a normal subring  when any maximal ideal in $A$ restricts to an $A$-invariant ideal in $B$ \cite{R}.
We show in Section~4 that if $M$ is the inclusion matrix
of a semisimple algebra pair $B \subseteq A$, the depth two
condition is equivalent to  $B$ being a normal subring in $A$.  As a consequence,  higher depth subgroups or semisimple subalgebras may be described
as being normal further along in the Jones tower of iterated endomorphism algebras (Corollary~\ref{cor-tower} below).

In \cite{K1} the notion of depth more
than two for a Frobenius extension $B \subseteq A$  is shown to be simplified  via a generalization of depth two from ring extensions to  towers of three rings, for a tower of three rings should be
 chosen in
the Jones tower above $B \subseteq A$. In \cite{BK} this idea was applied
to a pair of semisimple complex algebras $B \subseteq A$ with inclusion matrix $M$: the subalgebra pair is depth $n$ if
$n$ is least integer for which  $M^{n+1}$
is less than a multiple of $M^{n-1}$, where powers of an $r \times s$ matrix $M$
are understood by  $M^2 = MM^t$, $M^3 = MM^tM$, and so forth.  As
noted in Section~2 below, depth $n$ is equivalently the  point of stabilization of the zero entries
of even or odd powers of $M$, which form a descending chain of subsets.

In \cite{BK} the generalized depth two condition on a tower of semisimple
algebras $C \subseteq B \subseteq A$ with inclusion matrices $N$ and $\tilde{M}$,
respectively, is given by $N\tilde{M}\tilde{M}^t\tilde{M} \leq qN\tilde{M}$ where $q$ is a positive
integer: let $N$ be the identity matrix to recover the depth two condition on
a subalgebra pair. Build a  tower of algebras above $B \subseteq A$,
where $A_1 = \End A_B$ and one iterates the endomorphism ring construction
and embeds via left multiplication. Then $A \hookrightarrow A_1$
 has inclusion matrix $M^t$, and subalgebra pair $B \subseteq A$ is depth $n$
if $n$ is the least integer for which the tower $B \subseteq A_{n-3} \subseteq A_{n-2}$ satisfies the
generalized depth two condition. For this tower of three algebras the  inclusion matrices
are $N = M^{n-2}$, $\tilde{M} = M$ or $M^t$, in the generalized depth two condition, which plugged in and simplified, becomes the depth $n$
condition, $M^{n+1} \leq qM^{n-1}$, on the  inclusion matrix $M$.

The paper is organized as follows.  In Section~1
we define a matrix $M$ of non-negative integer coefficients with nonzero rows and columns to be depth $n > 1$ if $n$ is the least integer for which the $n+1$'st power of $M$ is less than
a multiple of the $n-1$'st power of $M$, where $M^2$ denotes $MM^t$,
$M^3 = MM^tM$ and so on. For example, if $M$ is the induction-restriction
table of irreducible characters of a finite group $G$ and a subgroup $H$,
then $M$ is depth two if and only if   $H$ is a normal subgroup of $G$ \cite{KK}.
If $M$ is the inclusion matrix of a subalgebra pair of multimatrix $\C$-algebras
$B \subseteq A$,  the matrix $M$ is depth two if and only if
$B$ is a normal subring in $A$
(in the sense of Rieffel \cite{R}) as shown below in Theorem~\ref{th-Rief}.  We
study depth three or more in Section~5,~6 and several appendices by Danz and the third coauthor. We prove that the induction-restriction
depth of permutation groups $S_{n} < S_{n+1}$ is $2n-1$. In Section~3, we make use of a well-known interpretation of the inclusion matrix $M$ as the incidence matrix of a bicolored
weighted multigraph of semisimple algebras $B \subseteq A$,
showing that odd depth is one plus the diameter of the row corresponding
to the simples of $B$; even depth is two plus the maximum graphical  distance along edges of the graph from an equivalence class of simples of $B$,
under one simple of $A$, to the simple of $B$
furthest away.   In Section~6, we show that a subgroup $H$ of $G$ has depth
bounded above by $2n$ if the largest normal subgroup of $G$ contained in $H$ (i.e.\ the core) is the intersection of $n$ conjugates
of $H$ (and bounded above by $2n-1$ if the core is trivial).

Throughout this paper we work over $\C$, which may be replaced by any algebraically closed field of characteristic zero with the same results.
In this case semisimple algebras  split into direct products of matrix algebras which we call multimatrix algebras. In this paper then, semisimple
algebra should be understood as split semisimple or multimatrix algebra.

\section{Depth of an irredundant matrix}

In this section, we make an introductory study of irredundant matrices,
which naturally arise as the induction-restriction table of irreducible $\C$-characters of a subgroup within a finite
group.  This type of matrix also  occurs more generally as the induction-restriction table of simple modules, equivalently incidence matrices
of inclusion diagrams, belonging to a subalgebra pair of  semisimple algebras as explained in the next section.

Let $M = (m_{ij})$ be an $r \times s$ matrix of non-negative integer entries $m_{ij} \in \Z_{\geq 0}$,
where each column and row vector is nonzero; such a matrix is called
\textit{irredundant} matrix in this paper. The matrix $M$ is \textit{positive}
if each of its entries is a positive integer.
Its (right) \textit{square} will be
the order $r$ symmetric matrix $\mathcal{S} := M^2 := MM^t$.
Note that the $(i,j)$-entry $s_{ij}$ of $\mathcal{S}$
is the euclidean inner product of rows $i$ and $j$ in $M$.
In particular, the diagonal $s_{ii} > 0$ since each row in $M$ is nonzero.

Continuing, the cube of $M$ is just $M^3 = MM^tM = \mathcal{S}M$, $M^4 = \mathcal{S}^2$, etc.
The odd powers $M^{2n+1} = \mathcal{S}^nM$ are all of size $r \times s$,
and the even powers $M^{2n} = \mathcal{S}^n$ are symmetric matrices of order $r$.

Let $N= (n_{ij})$ be another $r \times s$ irredundant matrix.  The matrices $N$ and $M$ are \textit{equivalent up to permutation} if there are permutation matrices $P \in S_r$
and $Q \in S_s$ such that $M = PNQ$.  We use the ordering
$M \geq N$ if $m_{ij} \geq n_{ij}$ for each $i = 1,\ldots,r$ and $j= 1,\ldots,s$.  If $T \geq 0$ denotes a square irredundant matrix of order $r$,
then $TM \geq TN$ if $M \geq N$; similarly $MT \geq NT$ for an $s \times s$-matrix
$T \geq 0$. Note that if $T$ is positive, i.e. $T > 0$, then also $TM > 0$
since columns in $M$ include positive entries.

The definition below comes from considerations of what  higher depth is
for Frobenius extensions when restricted to semisimple algebra pairs with inclusion matrix $M$ as outlined in the introduction. (Although not needed in this paper, the interested reader
should see \cite{KN, K1} for the definition of higher depth Frobenius extensions and see \cite{BK} for  why semisimple $\C$-algebra pairs
are split separable Frobenius extensions.)

\begin{definition} An $r \times s$ matrix $M$ is of \it{depth} $n \geq 2$ if $n$ is the least integer for
which  the following inequality (called a depth $n$
matrix inequality) holds for some
$q \in \Z_+$,
\begin{equation}
 M^{n+1} \leq qM^{n-1}.
\end{equation}
\end{definition}

The definition depends only on the equivalence class
of $M$ up to permutation.  In the case $n = 2$, $M$ is depth two if $\mathcal{S}M \leq qM$ for some positive integer $q$.
Multiplying from the right by $M^t$, one obtains $\mathcal{S}^2 \leq q\mathcal{S}$
(fix the notation $\mathcal{S} = MM^t$); thus, $M$ also satisfies a  depth three condition.  Iterating,  we see $M$ satisfies depth $n \geq 2$ conditions, and similarly, a depth $n$ matrix satisfies depth $n + k$ matrix inequality for any positive integer $k$.
However, the depth $n$ matrix inequality for $M$ does not imply a  depth $n-1$ matrix inequality for $M$. We  note
at several points in this paper that $M$ necessarily has finite depth.
Again, the depth of $M$ is the least positive integer $k \geq 2$ for which $M$
satisfies a depth $k$ inequality.

Given an integer irredundant
$r \times s$ matrix $M = (m_{ij})$, let the zero entries be
collected in $Z(M) = \{ (i,j) | m_{ij} = 0 \}$; its complement
$A(M) = \{ (i,j) | m_{ij} \neq 0 \}$ will be useful in the next section.
Note that the zero entries of the even or odd powers of an irredundant matrix $M$ form
a descending chain of subsets of the rectangular array $\underline{r} \times \underline{s}$ where $\underline{r} = \{1,\ldots, r\}$:
$$ Z(M^{n-1}) \supseteq Z(M^{n+1}) \supseteq  Z(M^{n+3}) \supseteq \cdots $$
The proof is simply the following.
Let $n \geq 2$ and denote the entries of  $M^{n+1}$ by $ (p_{ij})$ and the entries of $M^{n-1}$ by $(q_{ij})$. Suppose that
$0 = p_{ij} = \sum_{k=1}^r s_{ik}q_{kj}$. Then $q_{ij} = 0$
since the diagonal entries of $M^2$ are positive, so $s_{ii} > 0$.

From the finite order of the matrix $M$ it is clear that the descending chain of  subsets $Z(M^{n + 2k+1})$ ($k = \ldots, -1,0,1,\ldots$)
must stabilize at some point (i.e. satisfy dcc).  The next proposition notes that the point
at which these subsets are equal is the depth of the matrix.

\begin{prop} An irredundant matrix $M$ satisfies a depth $n$ inequality if and only if
$$Z(M^{n-1}) = Z(M^{n+1}).$$
\end{prop}
\begin{proof} If $M$ is depth $n$, then there is a positive integer $t$ such
that $M^{n+1} \leq tM^{n-1}$. Suppose the entry $q_{ij} = 0$ in the matrix $M^{n-1}$. By the matrix inequality, the corresponding entry in $M^{n+1}$ is zero, so $Z(M^{n-1}) \subseteq Z(M^{n+1})$.  Together with the opposite inclusion noted above, conclude that $Z(M^{n-1}) = Z(M^{n+1})$.

Conversely, if $Z(M^{n-1}) = Z(M^{n+1})$,  we may choose $t$ to be the maximum of the integer entries in $M^{n+1}$, in which case $M^{n+1} \leq t M^{n-1}$.
\end{proof}
The nonzero entries in $M$ recorded in the subset $A(M)$  occur in the next section.  The subset $A(M)$ is of course the complement of $Z(M)$
and the like powers form an ascending chain,
$$ A(M^{n-1}) \subseteq A(M^{n+1}) \subseteq A(M^{n+3}) \subseteq \cdots $$
which must end at some point (satisfying an acc condition).
The following is an obvious corollary:
\begin{cor}
\label{cor-stabilization}
An irredundant matrix $M$ satisfies a depth $n$ inequality if and only if $A(M^{n-1}) = A(M^{n+1})$.
\end{cor}
Another obvious corollary:
\begin{cor} If $M^{n-1} > 0$, then $M$ has depth $n$ or less.
\end{cor}
As noted above, an irredundant matrix has finite depth. In fact,
\begin{cor} An irredundant $r \times s$ matrix $M$ has depth $\leq 2r-1$.
\end{cor}
\begin{proof} The characteristic polynomial of $\mathcal{S} = MM^t$ gives an equation of
the form
$$\mathcal{S}^r + a_1\mathcal{S}^{r-1} + \cdots + a_{r-1} \mathcal{S} + a_r = 0$$
with $a_1, \ldots, a_r \in \C$, by the Cayley-Hamilton theorem. Thus
$$A(M^{2r}) = A(\mathcal{S}^r) \subseteq A(\mathcal{S}^{r-1}) = A(M^{2r-2}),$$
so that $M$ has depth $2r-1$ or less.
\end{proof}
\begin{example}\normalfont
We compute the depth of the matrix $M$ below, which is the inclusion
matrix, equivalently induction-restriction table, of the subgroup
$D_8 < S_4$, the symmetries of the square in the group of all
permutations of four letters. One has
$$M= \left( \begin{array}{ccccc}
  1 & 0 & 1 & 0 & 0 \\
0 & 0 & 0 & 0 & 1 \\
0 & 0 & 0 & 1 & 0 \\
0 & 1 & 1 & 0 & 0 \\
0 & 0 & 0 & 1 & 1
\end{array} \right),
 $$
and it can be computed directly that $M$ has depth $4$ (indeed
$\mathcal{S}^2M \leq 5\mathcal{S}M$).
\end{example}

\section{On multimatrix algebra inclusions}
\label{mult}

Suppose $B \subseteq A$ is an inclusion of multimatrix algebras.
Label the simple $A$-modules by $V_1, \ldots, V_s$ and the simples of $B$ by
$W_1, \ldots, W_r$. Restrict the $k$'th
simple $A$-module $V_k$ to a $B$-module and express in terms of direct
sum of simples
\begin{equation}
V_k\!\downarrow_B \cong \oplus_{i = 1}^r m_{ik}W_i
\end{equation}
We let $M $ be the $r \times s$-matrix, or table, with entries $m_{ij}$: $M = (m_{ij})$. By a well-known generalization of Frobenius reciprocity, the rows give induction of the $B$-simples:
\begin{equation}\label{ind}
W_i \! \uparrow^A = W_i^A = \oplus_{j = 1}^s m_{ij}V_j
\end{equation}
since $W_j^A = W_j \otimes_B A$, $V_k \downarrow_B \cong \Hom (A_B,
V_k)$ and $\Hom (W_i, W_j) = \{ 0 \}$ if $i \neq j$; i.e., if $[W_j^A, V_k]$ denotes the number of constituents
in $W_j^A$ isomorphic to $V_k$, Frobenius reciprocity is given
by
\begin{equation}
[W_i^A, V_k] = m_{ik} = [W_i, V_k\! \downarrow_B]
\end{equation}

The matrix $M$ is known as the \textit{inclusion matrix} of $B$
in $A$ \cite{GHJ}. As seen above it corresponds to the induction-restriction table  (as it is known in group theory) for simples or their irreducible characters.
It may also be viewed as a matrix representation of a linear mapping in  K-theory, between the groups ${\Z}^r \cong K_0(B) \rightarrow K_0(A) \cong {\Z}^s$.

For group algebras $A$ and $B$ corresponding to a finite group $G$ with
subgroup $H$, we will in practice obtain $M$ as follows.  Both character tables of $G$ and $H$ will be assumed known.  Restrict each of the $s$ irreducible characters of $G$
to $H$, then express the restricted character of $G$ as a non-negative integer coefficient linear combination of the $r$ irreducibles of $H$ by using inner products of characters.

%sim and approx
\subsection{The relation and the equivalence relation}\label{dee}
We say that $W_i \sim W_j$ if and only if there is a simple $A$-module $V$ such that $W_i $ and $W_j$ are both constituents of $V\downarrow_B$. The relation $\sim$ is reflexive and symmetric but not transitive in general.

Its transitive closure is denoted by $\approx$. Thus we say that $W_i \approx W_j$ if and only if there is $m \geq 1$ and a sequence $W_{i_0}, W_{i_1},\cdots, W_{i_{m-1}}, W_{i_m}$ of simple $B$-modules such that $W_i=W_{i_0} \sim W_{i_1}\sim W_{i_2} \sim \cdots \sim W_{i_{m-1}} \sim W_{i_m}=W_j$. %(Of course $W_0,\cdots, W_{m} \in \{W_1, W_2, \cdots, W_r\}$).
%distance

{\bf Notation:} We denote the above equivalence relation by $d_B^A$.
(This equivalence relation is considered before in \cite[Rieffel]{R};
this section and the next may be viewed as a continuation of results
in this article.)

%By definition of $\sim$ one has that $\alpha \sim \beta$ if and only if $m_B(\alpha,\; T(\beta))>0$. Remark that $\alpha$ is a constituent of $T(\alpha)$ for any irreducible character $\alpha$. Also if $\alpha$ is a constituent of $\beta$ then clearly $T(\alpha)$ is a constituent of $T(\beta)$. If $\alpha \sim \beta$ and $ \beta \sim \gamma $ then $\alpha$ is a constituent of $T(\beta)$ and $\beta$ is a constituent of $T(\gamma)$. Therefore $\alpha$ is a constituent of $T^2(\gamma)$.
\subsection{On the equivalence relation}
Let $B$ be a multimatrix subalgebra of $A$.
For an irreducible character $\alpha \in \mathrm{Irr}(B)$ let $f_{\alpha}$ be the central idempotent corresponding to $\alpha$. Similarly if $\chi \in \mathrm{Irr}(A)$ then $e_{\chi}$ is the corresponding central idempotent in $A$.
Consider the commutative algebra $\mathcal{Z}(A)\cap B $ as a subalgebra of $\mathcal{Z}(A)$ and $\mathcal{Z}(B)$. Then there are partitions of characters $\mathrm{Irr}(A)=\bigsqcup_{i=1}^t\mathcal{A}_i$ and $\mathrm{Irr}(B)=\bigsqcup_{i=1}^t\mathcal{B}_i$ such that a basis of primitive idempotents for $\mathcal{Z}(A)\cap B $ is given by

\begin{equation}
\label{mformula}
m_i=\sum_{ \chi \in \mathcal{A}_i}e_{\chi}=\sum_{ \alpha \in \mathcal{B}_i}f_{\alpha}.
\end{equation}

\begin{prop}\label{div1} Suppose that $B\subset A$ is an inclusion of multimatrix algebras with $A$ free as left $B$-module.
With the above notations it follows that

\begin{equation}
\label{firstformula}
(\sum_{ \chi \in \mathcal{A}_i}\chi(1)\chi)\downarrow_B^A=\frac{\dim A}{\dim B}\sum_{ \alpha \in \mathcal{B}_i}\alpha(1)\alpha.
\end{equation}

\begin{equation}
\label{secondformula}
(\sum_{ \alpha \in \mathcal{B}_i}\alpha(1)\alpha)\uparrow^A_B=\sum_{ \chi \in \mathcal{A}_i}\chi(1)\chi.
\end{equation}
\end{prop}
\begin{proof}
For each partition set $\mathcal{A}_i$ let $\chi_i=\sum_{ \chi \in \mathcal{A}_i}\chi(1)\chi$. The regular character $\rho_A$ of $A$ is given by $\rho_A=\sum_{i=1}^t\chi_i$. Since $A$ is free as left $B$-module it follows that $\rho_A\downarrow_B^A=\frac{\dim A}{\dim B}\rho_B$ where $\rho_B$ is the regular character of $B$. Similarly for each partition set $\mathcal{B}_i$ let $\alpha_i=\sum_{ \alpha \in \mathcal{B}_i}\alpha(1)\alpha$. Therefore $\rho_B=\sum_{i=1}^t\alpha_i$. From  Eq.~(\ref{mformula}) it follows that if $i\neq j$ then $\chi_i(m_j)=0$. This implies that $\chi_i(f_{\alpha})=0$
for all $j \neq i$ and $\alpha \in \mathcal{B}_j$. Since $\chi_i(f_{\alpha})=\alpha(1)m_B(\chi_i\downarrow^A_B,\;\alpha)$ one gets Eq.~(\ref{firstformula}). Since $\rho_B\uparrow^A_B=\rho_A$, Eq.~(\ref{secondformula}) follows using Frobenius reciprocity and Eq.~(\ref{firstformula}).
\end{proof}
\subsection{}\label{you}
Define a relation on $\mathrm{Irr}(A)$ by $\chi \sim \mu$ if $\chi\downarrow^A_B$ and $\mu\downarrow^A_B$ have a common constituent. This is not transitive but we can take the transitive closure. Thus $\chi \approx \mu $ if and only if there are irreducible $A$-characters $\mu_0,\cdots,\mu_r$ such that $\chi=\mu_0\sim \mu_1\sim\cdots \sim \mu_r=\mu$. This is an equivalence relation denoted by $u^A_B$.

\subsection{}
Let $\alpha$ be an irreducible $B$-character. 
Then all the irreducible constituents of $\alpha\uparrow^A_B$ are in the same equivalence class of $u^A_B$ (they are $\sim$ related). 
Let $\alpha \sim \beta$ be two irreducible characters of $B$. 
This means that there is an irreducible $\chi$ of $\mathrm{Irr}(A)$ such that $\alpha,\; \beta$ are irreducible constituents of $\chi\downarrow^A_B$. 
By Frobenius reciprocity $\chi$ is a common constituent for $\alpha\uparrow^A_B$ and $\beta\uparrow^A_B$. 
Thus all the irreducible constituents of $\alpha\uparrow^A_B$ and $\beta\uparrow^A_B$ are $u^A_B$-equivalent. This implies by transitivity that if $\alpha \approx \beta$ then all the irreducible constituents of $\alpha\uparrow^A_B$ and those of $\beta\uparrow^A_B$ are equivalent.

\subsection{} A similar argument to the one just given shows that  characters equivalent by $u^A_B$ are under restriction of characters equivalent by $d^A_B$.

\begin{prop}\label{samepart}Suppose that $B\subset A$ is an inclusion of multimatrix algebras with $A$ free as left $B$-module.
Then the equivalence classes of $u^A_B$ are the sets $\mathcal{A}_i$. The equivalence classes of $d_B^A$ are the sets $\mathcal{B}_i$.
\end{prop}

\begin{proof}
The previous two subsections imply that an equivalence class of $d_B^A$ induced to $A$ has all the constituents inside an equivalence class of $u_B^A$. Conversely, any equivalence class on $u_B^A$ restricted to $B$ has all the constituents inside an equivalence class of $d_B^A$. Since the regular character of $A$ is a sum of copies of the regular character of $B$,  it follows that the above restriction covers entirely the equivalence class. This implies that for any equivalence class $\mathcal{C}$ of $d^A_B$ the idempotent $\sum_{\alpha \in \mathcal {C}}f_{\alpha}$ is central in $A$. This finishes the proof.
\end{proof}

\begin{cor} \label{dim}Suppose that $B\subset A$ is an inclusion of multimatrix algebras with $A$ free as left $B$-module. Then the number of equivalence classes of $u^A_B$ is the same as the number of equivalence classes of $d^A_B$ and it equals $\mathrm{dim}\, (\mathcal{Z}(A)\cap B)$.
\end{cor}
\subsection{Distance between modules (characters)}
We say that the distance $d(W_i,\;W_j)$ between two modules is $m$ if $m$ is the smallest number such that there are $m-1$ intermediate simple $B$-modules with $W_i=W_{i_0} \sim W_{i_1}\sim \cdots \sim W_{i_m}=W_j$. Thus $d(W_i,\;W_j)=1$ if and only if $W_i \sim W_j$. We put $d(W_i,\;W_j)=-\infty$ if $W_i$ and $W_j$ are not equivalent under $\approx$ and $d(W_i ,\;W_i)=0$ for all $1\leq i \leq r$. Note that the distance defined here is half of the \textit{graphical distance} between the black points corresponding to $W_i$ and $W_j$ in the Bratteli diagram.

Recall the inclusion matrix $M= (m_{ij})$ for the inclusion of multimatrix algebras $B \subseteq A$ and define its
symmetric `square' $\mathcal{S}=MM^t = (s_{ij})$.  The entries of $\mathcal{S}$
and its powers, or indeed the powers $\mathcal{S}^mM$ of $M$ in the sense of the previous section, will be denoted by $(\mathcal{S}^mM^k)_{ij}$
where $k = 0,1$.
%matrix $\mathcal{S}$ entries

\begin{remark}\label{entries1}\normalfont
Let $1 \leq i \leq r$ and $1 \leq u \leq s$. Then $m_{iu}>0$ if and only if $W_i$ is a constituent of $V_u\downarrow_B$.
\end{remark}

\begin{prop}\label{prop-entries2}
Suppose $i \neq j$. One has $(\mathcal{S}^m)_{ij}>0$ if and only if $0 <d(W_i,\;W_j)\leq m.$ This is equivalent to the existence of a path of length $2m$ between $W_i$ and $W_j$.
\end{prop}

\begin{proof}
Observe first that $s_{ij}>0$ if and only if $W_i \sim W_j$ or equivalently $d(W_i,\;W_j)=1$. Indeed $s_{ij}=\sum _{u=1}^sm_{iu}m_{ju}$. Thus $s_{ij}>0$ if and only if there is $u$ such that $m_{iu}>0$ and $m_{ju}>0$. That means that $W_i$ and $W_j$ are constituents of $V_u\downarrow_B$ and therefore $W_i \sim W_j$.

For $m>1$ note that $(\mathcal{S}^m)_{ij}=\sum_{l_1,\cdots, l_{m-1}}s_{il_1}s_{l_1l_2}\cdots s_{l_{m-1}j}$. Thus $(\mathcal{S}^m)_{ij}>0$ if and only if there are $1 \leq l_1,\cdots, l_{m-1} \leq r$ such that $W_i \sim V_{l_1}\sim  \cdots \sim V_{l_{m-1}} \sim W_j$, i.e. $d(W_i,\;W_j)\leq m$.
\end{proof}

\begin{remark} \label{nonzeros}\normalfont
For $M \in \mathcal{M}_{rs} (\R_{\geq 0})$ let $A(M)$ be the set of ordered pairs $(i,u)$ with $m_{iu}>0$. Recall that given two matrices $M, N \in \mathcal{M}_{rs} (\R_{\geq 0})$ then there is $q \in \Z_+$ such that $M\leq qN$ if and only if $A(M)\subseteq A(N)$.
\end{remark}

\begin{remark}\normalfont
 \label{supersets}
We recall a few things from Section~2.
Note that $\mathcal{S}_{ii}>0$ and therefore $(\mathcal{S}^p)_{ii}>0$ for all $p>0$. This implies that if $(\mathcal{S}^m)_{ij}>0$ then also $(\mathcal{S}^{m+p})_{ij}>0$ since $(\mathcal{S}^{m+p})_{ij}=\sum_{l=1}^r(\mathcal{S}^m)_{il}(\mathcal{S}^p)_{lj}$ and for $l=j$ both terms are positive. In terms of distance in the Bratteli diagram this is equivalent to the fact that if there is a path between $W_i$ and $W_j$ of length $2m$ then there is also a path of length $2(m+p)$ between the same points. For example, one can travel $p$-times back and forth along the last edge of the path of length $2m$. Using the notations from the previous remark this implies that $A(\mathcal{S}^m) \subseteq A(\mathcal{S}^{m+1})$ for all $m >0$.
\end{remark}

\begin{definition}
The depth of a multimatrix algebra inclusion $B \subseteq A$ is defined
to be the depth of its inclusion matrix $M$ (in terms of Section~2,
which also notes the definition is independent of the ordering in the basis of simples). A subgroup $H$ of a finite group $G$ is said to be
depth $n$ if the corresponding group algebras over $\C$ have
multimatrix algebra inclusion (via Maschke and Wedderburn theory) of depth $n$.
\end{definition}

The background for this definition is given in \cite{BK, KN, KS, KK, K1,K2}
and their bibliographies; the definition coincides with the definition of
depth introduced briefly in \cite[Burciu-Kadison]{BK}. For example, given group algebras with inclusion matrix $M$, where
$B = \C[H]$, $A = \C[G]$, subgroup $H $ of finite group $G$, we easily see
that the $(i,j)$-entry in the depth two condition $M^3 \leq qM$  is the same as the condition for depth two in \cite[Section 3]{KK} as  follows. Suppose the bases of irreducible characters are given by
$\mathrm{Irr}(G) = \{ \chi_1,\ldots,\chi_s\}$ and $\mathrm{Irr}(H) = \{ \psi_1,\ldots,\psi_r \}$.  Then $m_{ij} = \bra  \psi_i \, | \, (\chi_j)_H \ket$ and by Frobenius reciprocity $m_{ij} = \bra \psi_i^G \, | \, \chi_j \ket$.  The entries of $\mathcal{S}$ are then
$s_{ij} = \bra \psi_i^G \, | \, \psi_j^G \ket$,
since this entry is the inner product of the rows $i$ and $j$ of $M$.
Then we apply reciprocity and orthogonal expansion:
\begin{eqnarray}
 (\mathcal{S}M)_{ij} & = & \sum_k \bra \psi_i^G \, | \, \psi_k^G \ket \bra \psi_k \, | \, (\chi_j)_H \ket \nonumber \\
& = & \sum_k \bra (\psi_i^G)_H \, | \, \psi_k\ket \bra \psi_k \, | \, (\chi_j)_H  \ket \nonumber \\
& = & \bra (\psi_i^G)_H \, | \, (\chi_j)_H \ket = \bra ((\psi_i^G)_H)^G\, | \, \chi_j \ket
\end{eqnarray}

\begin{theorem}\label{odddepth}
The inclusion matrix of $B \subseteq A$ satisfies a depth $2m+1$ inequality ($m \geq 1$) if and only if the distance between any two simple $B$-modules is at most $m$.
\end{theorem}

\begin{proof}
By definition, (cf.\ \cite[Theorem 2.5]{BK})
the inclusion matrix $M$ of $B \subseteq A$ satisfies a depth $2m+1$ inequality iff $\mathcal{S}^{m+1}\leq q\mathcal{S}^{m}$ for some $q>0$ and $\mathcal{S} = MM^t$.

Suppose that $B \subseteq A$ is of depth $2m+1$.
By Corollary~\ref{cor-stabilization}, $A(\mathcal{S}^m) = A(\mathcal{S}^{m+p})$ for all $p > 0$.  If $(\mathcal{S}^{m+p})_{ij} > 0$, then $(\mathcal{S}^m)_{ij} > 0$, so $d(W_i, W_j) \leq m$ by Prop.~\ref{prop-entries2}.
It follows that $d(W_i,W_j) \leq m$ for all pairs of simples
$W_i, W_j$ over $B$ (the distance is $-\infty$ if two simples are not in the same connected component of the inclusion diagram).

Conversely, suppose that the distance between any two simple $B$-modules is at most $m$. We have to show that $\mathcal{S}^{m+1}\leq q\mathcal{S}^{m}$ for some $q>0$ which by Remark \ref{nonzeros} is equivalent to $A(\mathcal{S}^{m+1}) \subseteq A(\mathcal{S}^{m})$. If $\mathcal{S}^{m+1}_{ij}>0$ then in the Bratteli diagram there is a path of length $2m$ between $W_i$ and $W_j$. Therefore the distance between these two points is positive (not $-\infty$) and by the assumption it should be less or equal to $m$. Thus $\mathcal{S}^{m}_{ij}>0$ by Proposition \ref{prop-entries2}.
\end{proof}

If we define the \textit{diameter} of a row of simples in an inclusion diagram to be the greatest graphical distance between simples (an even number), the theorem
says that the minimum odd depth inequality satisfied by the
inclusion matrix of $B \subseteq A$ is one plus the diameter of the simples of $B$ in its inclusion diagram.

\begin{cor}\label{d3}
The inclusion $B \subseteq A$ is of depth $\leq 3$ if and only if $\sim$ is an equivalence relation.
\end{cor}

\begin{proof}
Suppose that $B \subseteq A$ is of depth $\leq 3$. By the above theorem the distance between any two modules is less or equal to $1$. Thus this distance is either $-\infty$ or $1$. If $W_i \sim W_j$ and $W_j \sim W_k$ then $0<d(W_i,\;W_k)\leq 2$. The assumption implies $d(W_i,\;W_k)=1$, i.e. $W_i \sim W_k$. This proves that $\sim$ is transitive.

Conversely suppose $\sim$ is transitive. Then this implies that the distance between any two modules is less or equal to $1$. It is $1$ if the modules are equivalent under $\sim$ and $-\infty$ if they are not equivalent. The above theorem implies that $B \subseteq A$ is of depth $\leq 3$.
\end{proof}

Let $V_u$ be a simple $A$-module with $1\leq u \leq s$. The irreducible constituents of $V_u\downarrow_B$ are all inside of one equivalence class of $\approx$. Denote the set of these constituents by $\mathcal{V}_u$. The distance between a simple $B$-module $W_i$ and the set $\mathcal{V}_u$ is defined as usually, by the minimal distance between $W_i$ and any element of the set $\mathcal{V}_u$. Thus \begin{equation} d(W_i,\; \mathcal{V}_u)=\min_{ _{W_j \in \mathcal{V}_u}}d(W_i,\;W_j).\end{equation}

\begin{definition}\label{defn}
We define $m(V_u)$ to be the maximal distance between any simple $B$-module $W_i$ and the set $\mathcal{V}_u$.
\end{definition}

Note that remark \ref{entries1} can now be written as $m_{iu}>0$ if and only if $W_i \in \mathcal{V}_u$.

\begin{prop}\label{stimesm}
Let $m \geq 1$. Then $(\mathcal{S}^mM)_{iu}>0$ if and only if $0<d(W_i,\; \mathcal{V}_u)\leq m$.
\end{prop}

\begin{proof}
Suppose that $(\mathcal{S}^mM)_{iu}>0$. Since $(\mathcal{S}^mM)_{iu}=\sum_{l=1}^r(\mathcal{S}^m)_{il}m_{lu}$ there is an $1 \leq l \leq r$ such that $(\mathcal{S}^m)_{il}>0$ and $m_{lu}>0$. Proposition~\ref{prop-entries2} and Remark~\ref{entries1} imply that $d(W_i,\;W_l)\leq m$ and $W_l \in \mathcal{V}_u$. Thus $0<d(W_i,\; \mathcal{V}_u)\leq m$.

Conversely, if $0<d(W_i,\; \mathcal{V}_u)\leq m$ then there is an $1 \leq l\leq r$ such that $d(W_i,\; W_l)\leq m$ and $W_l \in \mathcal{V}_u$. This implies that $(\mathcal{S}^m)_{il}>0$ and $m_{lu}>0$ which together give that $(\mathcal{S}^mM)_{iu}>0$.
\end{proof}

%\begin{remark} Similar to Remark \ref{supersets} it can be shown that $A(\mathcal{S}^mM) \subseteq A(\mathcal{S}^{m+1}M)$ for all $m \geq 1$.
%\end{remark}
\begin{theorem}
\label{evendepth}
The inclusion matrix of $B \subseteq A$ satisfies a depth $2m$ inequality (with {\bf $m \geq 2$}) if and only if $m(V_u)\leq m-1$ for any simple $A$-module $V_u$.
\end{theorem}

\begin{proof}
Recall that the inclusion matrix $M$ satisfies a depth $2m$ inequality iff $\mathcal{S}^mM\leq q\mathcal{S}^{m-1}M$.

Suppose that the inclusion is depth $2m$. Therefore $\mathcal{S}^mM\leq q\mathcal{S}^{m-1}M$ for some $q >0$. By induction one can prove that $\mathcal{S}^{m+p}M\leq q^{p+1}\mathcal{S}^{m-1}M$ (multiplying with $\mathcal{S}$ to the left). Suppose that $m(V_u) = m+p$ with $p \geq 0$ and some $u$. This implies that there is $W_i$ a simple $B$-module such that its distance to $\mathcal{V}_u$ is $m+p$. Proposition \ref{stimesm} implies that $(\mathcal{S}^{m+p}M)_{iu}>0$. From remark \ref{nonzeros} one has $A(\mathcal{S}^{m+p}M)\subseteq q^{p+1}A(\mathcal{S}^{m-1}M)$. Thus $(\mathcal{S}^{m-1}M)_{iu}>0$ and Proposition \ref{stimesm} implies that $d(W_i,\;\mathcal{V}_u)\leq m-1$. This is a contradiction and the proof in one direction is completed.

Conversely suppose that $m(\mathcal{V}_u)\leq m-1$ for any simple $A$-module $V_u$. We have to show that $\mathcal{S}^mM\leq q\mathcal{S}^{m-1}M$ which by Remark \ref{nonzeros}
is equivalent to $A(\mathcal{S}^mM)\subseteq A(\mathcal{S}^{m-1}M)$. Suppose $(i,\;u) \in A(\mathcal{S}^{m}M)$. Then $(\mathcal{S}^mM)_{iu}>0$ and Proposition \ref{stimesm} implies that $ 0<d(W_i ,\; \mathcal{V}_u)\leq m$. The assumption of the theorem implies that this distance should be less or equal to $m-1$. Then Proposition \ref{stimesm} implies that $(\mathcal{S}^{m-1}M)_{iu}>0$. Thus $(i,\;u) \in A(\mathcal{S}^{m-1}M)$.
\end{proof}

We note that in terms of graphical distance, the minimal even depth matrix inequality satisfied by the  inclusion diagram of $B$ in $A$ is two plus the largest graphical distance of a $B$-simple from an equivalence class
of $B$-simples under one $A$-simple.

\begin{example}\normalfont
The inclusion (or Bratteli) diagram starting with $B= \C[S_2] \subset A = \C[S_3]$
at the bottom level, and proceeding to its semisimple pair $A \hookrightarrow \End A_B$ via $\lambda$ at the top level is shown below:
%$$
%\begin{array}{rcccccccl}
% & & \stackrel{3}{\circ} & & & & \stackrel{3}{\circ}  & & \\
%&&&&&&&& \\
%& / & & \setminus & & / & &  \setminus  & \\
%&&&&&&&& \\
% \stackrel{1}{\bullet} & & &&\stackrel{2}{\bullet} & & & &\stackrel{1}{\bullet} \\
%&&&&&&&& \\
%& \setminus & & / & & \setminus & &  /  & \\
%&&&&&&&& \\
% & & \stackrel{\circ}{\mbox{\scriptsize 1}} & & & & \stackrel{\circ}{\mbox{\scriptsize 1}}  & &
%\end{array} $$
\[
\begin{xy}
\xymatrix{
& \overset{\displaystyle 3}{\circ} \ar@{-}[ld] \ar@{-}[rd] &&
\overset{\displaystyle 3}{\circ} \ar@{-}[ld] \ar@{-}[rd]&\\ \mathop{\bullet}
\ar@{-}[dr]_(.15){\displaystyle 1} & &\mathop{\bullet} \ar@{-}[ld]
\ar@{-}[rd]_(.15){\displaystyle 2}&
&\mathop{\bullet}\ar@{-}[ld]^(.15){\displaystyle 1}\\ &
\mathop{\circ}\limits_{\displaystyle 1} & &
\mathop{\circ}\limits_{\displaystyle 1}&}
\end{xy}
\]

Notice that the inclusion diagram of $A \hookrightarrow E$ is the reflection
of the diagram of $B \subseteq A$ about the middle row, true in general by Morita theory \cite{GHJ}.  Applying Theorem~\ref{odddepth}, we see from the bottom graph that the graphical
distance between simples is $2$, so depth of subgroup $S_2 < S_3$
is three.  Applying Theorem~\ref{evendepth}, we see from the top graph
that the maximal distance from a simple away from a  set of two simples in $\mathcal{V}_u$ on the middle line has  graphical distance $2$, so that the depth of $A \hookrightarrow E$ is four.

By simply adding dots and the same pattern of edges to the right of the diagram, we create diagrams
(Dynkin diagrams of type $A_n$) for multimatrix algebra inclusions of arbitrary odd or even depth.  In terms of explicit inclusion mappings, the following
inclusion $B := \C^n \rightarrow A := \C \times M_2(\C)^{n-1} \times \C$
has depth $2n-1$: ($\lambda_i \in \C$, $n \geq 2$)
$$(\lambda_1, \ldots, \lambda_n) \mapsto (\lambda_1, \left( \begin{array}{cc}
\lambda_1 & 0 \\
0 & \lambda_2
\end{array} \right),  \left( \begin{array}{cc}
\lambda_2 & 0 \\
0 & \lambda_3
\end{array} \right),\ldots, \left( \begin{array}{cc}
\lambda_{n-1} & 0 \\
0 & \lambda_n
\end{array} \right), \lambda_n ) $$
while its endomorphism algebra extension $A \hookrightarrow E = M_3(\C) \times M_4(\C)^{n-2} \times M_3(\C)$ has depth $2n$: ($M_i \in M_2(\C)$)
$$ (\lambda_1,M_1,\ldots,M_{n-1}, \lambda_n) \mapsto (\left( \begin{array}{cc}
\lambda_1 & 0 \\
0 & M_1
\end{array} \right), \left( \begin{array}{cc}
M_1 & 0 \\
0 & M_2
\end{array} \right),\ldots,\left( \begin{array}{cc}
M_{n-1} & 0 \\
0 & \lambda_n
\end{array} \right) )$$
\end{example}
\bigskip
\begin{remark}\normalfont
The definition of depth may be extended to the case depth one as follows.
Define $M^0$ to be the $r \times r$ identity matrix $I$ in the depth $n$ matrix inequality condition, in which case a depth one extension of semisimple algebras $B \subseteq A$ with inclusion matrix $M$
satisfies $\mathcal{S} \leq nI$ for some positive integer $n$.  This is satisfied by a centrally projective ring extension $B \subseteq A$, defined by  ${}_BA_B \oplus * \cong {}_BB^n_B$ for some $n$,
or equivalently there are $r_i \in C_A(B)$ and $f_i \in \Hom ({}_BA_B, {}_BB_B)$ such that each $a \in A$ satisfies
$a = \sum_{i=1}^n r_i f_i(a)$.  If $A$ and $B$ are the group $\C$-algebras corresponding to $G \geq H$, a depth one
extension is for example any subgroup of the center of $G$, or $H$ is normal in $G$ with a normal complement.
\end{remark}

\subsection{Endomorphism ring theorems for depth}
We continue our study of irredundant matrices from the point of view
of depth. If $M$ is the inclusion matrix of a subalgebra pair of
semisimple algebras $B \subseteq A$, then its transpose irredundant
matrix $M^t$ is the inclusion matrix of $A \hookrightarrow \End A_B$
(via $a \mapsto \lambda_a$ where $\lambda_a(x) = ax$ for every $a,x
\in A$) by an argument that goes as follows.  It is clear that the
natural module $A_B$ is finitely generated projective; it is indeed
also a generator since the ground field has characteristic zero.
Thus $B$ and $E := \End A_B$ are Morita equivalent algebras with
context bimodules ${}_EA_B$ and $A^* := {}_B\Hom (A_B,B_B)_E$; the
$E$-simples  are then $A \otimes_B W_i$ $(i = 1,\ldots,r)$.
Restricting the $E$-simples down to $A$ and using Eq.~(\ref{ind}),
the columns of the inclusion matrix of $A \hookrightarrow E$ are the
rows of $M$.  We conclude that the inclusion matrix of $A
\hookrightarrow E$ is $M^t$.

  Thus, it is interesting to  compare the depths of $M$ and $M^t$ in the next purely matrix-theoretic theorem.

\begin{theorem} If an irredundant matrix $M$ has depth $n$,
 then $M^t$ has depth $\leq n+1$. If $n$ is even, then $M^t$ is moreover of depth $n$.
\end{theorem}
\begin{proof} If $M^{n+1} \leq q M^{n-1}$ for some $q \in \Z_+$,
we multiply from the left by $M^t$ to obtain $(M^t)^{n+2} \leq q
(M^t)^n$, which shows that $M^t$ has depth $\leq n+1$.

If $n$ is even then the transpose of the inequality $M^{n+1} \leq q
M^{n-1}$ is the inequality $(M^t)^{n+1} \leq q (M^t)^{n-1}$.
\end{proof}
Let $E$ denote $\End A_B$ and embed $A$ in $E$ via the mapping
$\lambda$ defined above.
\begin{cor}  The subalgebra pair of  semisimple algebras $B \subseteq A$
is of depth $2n$ if and only if its endomorphism algebra extension
$A \hookrightarrow E$ is of depth $2n$.  If $B \subseteq A$ is of
depth $2n-1$, then $A \hookrightarrow E$ is of depth $ \leq 2n$.
\end{cor}
This corollary is consistent with several general `endomorphism ring
theorems' in \cite{K1,K2} and is an improvement in the semisimple
case.
\begin{example}\normalfont
The matrix $M$ below, obtained from the inclusion matrix of $S_2 <
S_3$ (\cite{BK}), has depth three while its transpose has depth
four:
$$ M = \left(
\begin{array}{ccc}
1 & 0 & 1 \\
0 & 1 & 1
\end{array}
\right), \ MM^t = \left( \begin{array}{ccc}
2 & 1 \\
1 & 2
\end{array} \right), \ M^tM = \left( \begin{array}{ccc}
1 & 0 & 1 \\
0 & 1 & 1 \\
1 & 1 & 2
\end{array} \right) $$
This is easier to see graphically; we return to this example in the
next section.
\end{example}

It is also easy to see from the precise definition in the next
section that if $C \subseteq B$ and $B \subseteq A$ are successive
subalgebra pairs of semisimple algebras with inclusion matrices $M$
and $N$, respectively, then the inclusion matrix of the composite
subalgebra pair $C \subseteq A$ is of course $MN$.  As a simple
consequence we may note an improved  version of the embedding
theorem \cite[8.6]{K1}. We prove that any depth $n$ subalgebra pair
may be embedded in a depth two extension, depth two being an
improvement from the point of view of Galois theory (see \cite{KN,
KK, K1, K2} and papers in their bibliographies). We set up the
theorem by introducing the Jones tower above the subalgebra pair and
its endomorphism ring, $B \subseteq A \hookrightarrow E_1 := E =
\End A_B$.  The Jones tower is just obtained via iteration of the
right endomorphism ring construct:
\begin{equation}
B \subseteq A \hookrightarrow E_1  \hookrightarrow E_2
\hookrightarrow \cdots
\end{equation}
I.e., $E_2 = \End E_A$, and iterate with respect to $\lambda: E_1
\hookrightarrow E_2$ to form $E_3$, then continuing like this. Note
that $E_1$ is Morita equivalent to $B$, $E_2$ is Morita equivalent
to $A$ (the details are brought together in \cite[2.2]{BK}), so all
$E_m$'s are themselves semisimple algebras.  Then if $B \subseteq A$
has inclusion matrix $M$, $A \hookrightarrow E_1$ has inclusion
matrix $M^t$, and $E_1 \hookrightarrow E_2$ has again inclusion
matrix $M$, and so on in alternating fashion.

\begin{theorem}
\label{thm-tower}
  A depth $n$ subalgebra pair of semisimple algebras $B \subseteq A$
is embedded in  the depth two subalgebra pair of semisimple algebras
$B \hookrightarrow E_{n-2}$.
\end{theorem}
\begin{proof}
If the inclusion matrix of $B \subseteq A$ is $M$, then  $M^{n+1}
\leq q M^{n-1}$ for some $q \in \Z_+$.  Since $n \geq 2$, we have
$3n-3 \geq n+1$ and $M$  also satisfies a depth $3n-3$ matrix
inequality.  Then there is $r \in \Z_+$ such that $M^{3n-3} \leq
rM^{n-1}$.  In other words, by checking odd and even case, this is
the same as
$$M^{n-1}(M^{n-1})^t M^{n-1} \leq r M^{n-1} ,$$
which of course is the depth two condition for the matrix $M^{n-1} =
MM^tM\ldots$ ($n-1$ times $M$ and $M^t$ alternately). But $M^{n-1}$
is the inclusion matrix of the composite subalgebra pair $B
\hookrightarrow E_{n-2}$.
\end{proof}

\subsection{Depth of tensor product of matrices}  Given two irredundant
matrices, an $r \times s$ matrix $M= (m_{ij})$ and a $p \times q$
matrix $N= (n_{ij})$, we form the tensor product $M \otimes N$
corresponding to the tensor product of linear mappings between
vector spaces.  In terms of block matrix representation, $M \otimes
N$ is the $rp \times sq$ matrix $(m_{ij}N)$, or equivalently up to
permutation $(Mn_{ij})$. Our interest in determining the depth of $M
\otimes N$ knowing the depths of $M$ and of $N$ comes from the
following situation in group theory. Given a subgroup $H_1 < G_1$
with inclusion matrix $M$ and another subgroup $H_2 < G_2$ with
inclusion matrix $N$, the inclusion matrix of $H_1 \times H_2 < G_1
\times G_2$ is none other than $M \otimes N$.

\begin{prop}  Suppose irredundant matrix $M$ has depth $n$ and irredundant matrix
$N$ has depth $m$.  Then $M \otimes N$ has depth at most $\max \{ n,
m \}$.
\end{prop}
\begin{proof}  Suppose $n \geq m$.  Then we  show that $(M\otimes N)^{n+1} \leq
q (M \otimes N)^{n-1}$ for some $q \in \Z_+$.  Note that $(M \otimes
N)^t = M^t \otimes N^t$, then $(M \otimes N)^m = M^m \otimes N^m$ in
the meaning of power of non-square matrices given above.  But we are
given that $M^{n+1} \leq q_1 M^{n-1}$ for some $q_1 \in \Z_+$, and
since $N$ satisfies a depth $m$ matrix inequality  $\Rightarrow$ $N$
satisfies a depth $n$ matrix inequality, there is $q_2 \in \Z_+$
such that also $N^{n+1} \leq q_2 N^{n-1}$.  It  follows from the
definition of tensor product that $(M\otimes N)^{n+1} \leq q (M
\otimes N)^{n-1}$ with $q = q_1 q_2$.
\end{proof}
\section{Depth two and normality}

Let $S \subset R$ an inclusion of multimatrix algebras.
We define the subring $S$ to be \textit{normal} in $R$ if
the restriction of every maximal ideal $I$ (in $R$) to $S$ is $R$-invariant,
meaning that $(I \cap S)R = R(I \cap S)$ as subsets of $R$. This definition
of normal subrings is first given in \cite[Rieffel]{R}  and used to provide a ring-theoretic setting for Clifford theory. It is also
closely related historically to the $HK^+ = K^+H$ condition of normality of a Hopf subalgebra $K$ in a Hopf algebra $H$.

{\bf Notation: } Let $\hat{R}$ denote the set of maximal two sided ideals of $R$. Similarly define $\hat{S}$.
Any $I\in \hat{R}$ determines up to isomorphism a unique simple $R$-module denoted by $V_I$ and a minimal (primitive) central idempotent $f_I$ of $R$. Similarly any $J\in \hat{S}$ determines up to isomorphism a unique simple $S$-module denoted by $W_J$ and a minimal central idempotent $q_J$ of $S$.

\begin{prop}\label{const}
$W_J$ is a constituent of $V_I\downarrow_S$ if and only if $I\cap S \subset J$ if and only if $q_Jf_I\neq 0$.
\end{prop}
\begin{proof}
$W_J$ is a constituent of $V_I\downarrow_S$ if and only if $q_J(Rf_I)\neq 0$. Since $f_I$ is central this is equivalent with $q_Jf_I\neq 0$. The first statement results taking annihilators in $S$.
\end{proof}

\begin{prop}
1) For any idempotent $e \in S$ one has $$e=e(\sum_{\{I \in \hat{R}|\;ef_I\neq 0\}}f_I).$$

2) For any idempotent $f \in R$ one has $$f=f(\sum_{\{J \in \hat{S}|\;fq_J\neq 0\}}q_J).$$
\end{prop}

The next proposition makes an improvement of \cite[Prop.\ 2.10]{R}.

\begin{prop}\label{rfixed}
Assume that for any simple $A$-module $V$ the irreducible constituents of $V$ form an entire equivalence class of $\approx$. Then $\sim$ is transitive and $B$ is normal in $A$.
\end{prop}

\begin{proof}
Clearly $\sim$ is transitive. We follow the reasoning from Rieffel's proof of Proposition 2.10. Fix $J_0\in \hat{S}$. Construct $e_{J_0}=\sum_{I \in X_0}f_I$ where $X_0=\{I \in \hat{R}|q_{J_0}f_I\neq 0\}$. Then $q_{J_0}e_{J_0}=q_{J_0}$ by the first item of the above Proposition. Suppose that $q_Je_{J_0}\neq 0$ for some other $J \in \hat{S}$. Then $q_Jf_I\neq 0$ for at least one $I \in X_0$. Then $W_{J_0}\sim W_J$ by the above remark. The assumption on the restriction of modules implies that $q_Jf_I \neq 0$ for all $I \in X_0$ and $q_Jf_I= 0$ for all $I \notin X_0$. The first item of the above Proposition shows that $q_J=q_Je_{J_0}$. The second item of the same Proposition implies that

\begin{equation} e_{J_0} =\sum_{ \{J \in \hat{S}\;| e_{J_o}q_J \neq 0\}}e_{J_0}q_J=\sum_{\{J \in \hat{S}\;| e_{J_o}q_J \neq 0\}}q_J.\end{equation}
Thus $q_{J_0} \in S$. Then the rest of the proof is the same as in Rieffel's paper. Indeed, if $I \in \hat{R}$ such that $I\cap S\subseteq J_0$ then $1-e_{J_0} \in I\cap S$. On the other hand if $J \in \hat{S}$ and $I\cap S \nsubseteq J$ then $q_J(1-e_{J_0})=q_J$. Thus $1-e_{J_0}$ is the identity element of $I \cap S$. Since it is central in $R$, Proposition~2.6 of \cite[Rieffel]{R} implies that $I \cap S$ is $R$-invariant. Thus $S$ is normal in $R$.
\end{proof}

\begin{lemma}\label{d2}
$\mathcal{S}M \leq n M$ for some $n>0$ if and only if for any simple $A$-module $V$ the irreducible constituents of $V\downarrow_B$ form an entire equivalence class of $\approx$. \end{lemma}

\begin{remark}\normalfont
With the previous notations the statement of the lemma can be rephrased as $\mathcal{S}M \leq n M$ for some $n >0$ if and only if $\mathcal{V}_u$ coincides with an entire equivalence class of $\sim$ for any simple $A$-module $V_u$.
\end{remark}

\begin{proof}
Suppose that $\mathcal{S}M \leq n M$ for some $n>0$. Clearly the  nonzero-entry subsets satisfy $A(\mathcal{S}M) \subset A(M)$. Suppose that $W_i \sim W_j$ and $W_j \in \mathcal{V}_u$. Then Proposition~\ref{prop-entries2} and Remark \ref{entries1} imply that $s_{ij}>0$ and $m_{ju}>0$. Thus $(\mathcal{S}M)_{iu}>0$. This means $(i,u) \in A(\mathcal{S}M) \subset A(M)$. Thus $m_{iu}>0$, i.e. $W_i \in \mathcal{V}_u$.

Conversely, suppose that $\mathcal{V}_u$ coincides with an entire equivalence class of $\sim$. We need to show $\mathcal{S}M \leq n M$ for some $n>0$ or $A(\mathcal{S}M) \subset A(M)$.
Let $\{(i,u),\; (u,i)\} \in A(\mathcal{S}M)$. Thus there is $1 \leq l \leq r$ such that $\mathcal{S}_{il}>0$ and $m_{lu}>0$. This means $W_i \sim W_l$ and $W_l \in \mathcal{V}_u$. Since $\mathcal{V}_u$ coincides with an entire equivalence class of $\sim$ it follows that $W_i \in \mathcal{V}_u$. Thus $m_{iu}>0$ and $(i,\;u) \in A(M)$.
\end{proof}

\begin{theorem}
\label{th-Rief}
The inclusion $B \subseteq A$ is of depth $2$ if and only if $B$ is normal in $A$.
\end{theorem}

\begin{proof}
($\Rightarrow$) By Proposition~2.2 of \cite[Burciu-Kadison]{BK}, $B \subseteq A$ is depth $2$ if and only if $\mathcal{S}M \leq n M$ for some $n$. If $B \subseteq A$ is normal then from Proposition 2.8 in \cite{R} it follows that $\sim$ is an equivalence relation and for any simple $A$-module $V$ then the irreducible constituents of $V\downarrow_B$ form an entire equivalence class. Lemma \ref{d2} implies that the multimatrix algebra inclusion is depth two.

($\Leftarrow$) The converse follows from Lemma~\ref{d2} and Proposition~\ref{rfixed}.
\end{proof}

This theorem provides a third proof of \cite[main theorem, 5.1]{BK2}.
Recall that a  Hopf subalgebra $K$ of a Hopf algebra $H$ is \textit{normal}
if $HK^+ = K^+H$ for $K^+$ the kernel of the counit $\eps: K \rightarrow \C$.

\begin{cor}
A depth two Hopf subalgebra of a semisimple Hopf algebra is normal.
\end{cor}
\begin{proof}
A Hopf subalgebra of a semisimple Hopf algebra is known to be semisimple.
Since it is depth two, it is normal as defined by Rieffel.   But $H^+ \cap K = K^+$  for the kernel of the counit $\eps: H \rightarrow \C$, a maximal
ideal in $H$.  This ideal in $K$ is then $H$-invariant, so $K$ is a normal Hopf subalgebra.
\end{proof}
The following corollary puts the last theorem together with Theorem~\ref{thm-tower} in a previous section.
Recall that given a subalgebra $B \subseteq A$, the first
endomorphism ring $E_1 = \End A_B$, the second $E_2 = \End (E_1)_A$, and so forth (with $A = E_0$ and $B = E_{-1}$).
We embed $E_n$ in $E_{n+1}$ as before via left multiplication in what we call the \textit{Jones tower}
above $B \subseteq A$.
\begin{cor}
\label{cor-tower}
Let $B$ be a semisimple subalgebra in a semisimple algebra $A$. If $B$ is depth $n$
then $B$ is normal in $ E_{n-2}$.
\end{cor}
The converse may be shown as well by a theorem in \cite{CK}
and its generalization to higher depth. We have
then seen depth two subalgebra to be normal in the overalgebra, a depth three subalgebra to be normal in $E_1$  and a higher
depth subalgebra to be normal further along
in the Jones tower.
\section{Inclusions of semisimple Hopf algebras}
Let $K \subset H$ be an inclusion of semisimple Hopf subalgebras. Let $C(H)$ and $C(K)$ be the character rings of $H$ and $K$ respectively.  These are commutative rings if $H$ and $K$ are quasitriangular or cocommutative Hopf algebras such as group algebras.

If $M$ and $N$ are two $H$-modules with characters $\chi$ and $\mu$ respectively, then $m_H(M,\;N):=\mathrm{dim}\, \mathrm{Hom}_H(M,\;N)$. The same quantity is also denoted by $m_H(\chi,\;\mu)$. In this manner one obtains a nondegenerate symmetric bilinear form $m_H(\;,\:\;)$ on the character ring $C(H)$ of $H$.
The following result is Proposition 2 of \cite{Bd}. It shows that the image of the induction map is a two sided ideal in $C(H)$. A different proof that also works in the nonsemisimple case is given below.

\begin{lemma}\label{pr}
Let $K$ be a Hopf subalgebra of a semisimple Hopf algebra $H$. Let $M$ be an $H$-module and $V$ a $K$-module. Then
$$M\ot V\uparrow^H_K=(M\downarrow^H_K\ot V)\uparrow^H_K$$ and
$$V\uparrow^H_K\ot M =(V\ot M\downarrow^H_K)\uparrow^H_K$$
\end{lemma}

\begin{proof}
Define \begin{equation*}\phi: M\ot (H \ot_K V) \rightarrow H\ot_K(M\ot V)\end{equation*}
by $m\ot h\ot_K v\mapsto h_{(2)}\ot_K(S^{-1}(h_{(1)}) m\ot v)$.
It can be checked that $\phi$ is a well-defined morphism of $H$-modules. Moreover $\phi$ is bijective since $h\otimes_K\ot m\ot v\mapsto (h_{(1)}m\ot h_{(2)}\otimes_K v)$ is its inverse map.
\end{proof}

In terms of characters the first formula can be written as $\chi\mathrm{Ind}(\alpha)=\mathrm{Ind}( \mathrm{Res}(\chi)\alpha)$ or

\begin{equation}\label{eq2}
\chi\alpha\uparrow=(\chi\downarrow\;\alpha)\uparrow.
\end{equation}

In terms of characters the second formula can be written as $\mathrm{Ind}(\alpha)\chi=\mathrm{Ind}(\alpha \mathrm{Res}(\chi))$ or

\begin{equation}\label{eq}
\alpha\uparrow\chi=(\alpha\;\chi\downarrow)\uparrow.
\end{equation}

Let $T:C(K) \rightarrow C(K)$ be given by $T(\alpha)=\mathrm{Res}(\mathrm{Ind}(\alpha))$. Thus $T(\alpha)=\alpha\uparrow\downarrow$. Note that the matrix of the operator $T$ with respect to the basis of $C(K)$ given by the irreducible characters of $K$ is the matrix $\mathcal{S}$ defined in the previous section.

\begin{lemma}
With the above notations one has
\begin{equation} \label{albet}
T(\alpha T(\beta))=T(\alpha)T(\beta) = T(T(\alpha)\beta)
\end{equation}
for all $\alpha,\beta \in C(K)$.
\end{lemma}

\begin{proof} One has
$$T(\alpha T(\beta))=(\alpha(\beta\uparrow\downarrow))\uparrow\downarrow
=(\alpha((\beta\uparrow)\downarrow)\uparrow)\downarrow
=(\alpha\uparrow\beta\uparrow)\downarrow
=\alpha\uparrow\downarrow\beta\uparrow\downarrow
$$%=T(\alpha)T(\beta)
We have applied relation \ref{eq} for the fourth equality and the fact that $\mathrm{Res}$ is an algebra map in the last equality. So the first equation in
the lemma is proved, and the other is obtained in a similar way.
\end{proof}

\begin{prop} \label{formulae} With the above notations one has

1) $T^n(\eps)=T(\eps)^n$ for all $n\geq 1$.

2) $T^n(\alpha)=T(\alpha)T(\eps)^{n-1}$ for all $n\geq 1$.
\end{prop}
\begin{proof}
The first part is a special case of the second one. We prove the second by
induction on $n$. The case $n=1$ is trivial. Suppose that
$T^n(\alpha)=T(\alpha)T(\eps)^{n-1}$ for some $n$. Then
\begin{eqnarray*}
T^{n+1}(\alpha) &= &T(T^n(\alpha)) = T(T(\alpha)T(\eps)^{n-1}) = T(T(\alpha))
T(\eps)^{n-1}\\
&= &T(T(\alpha)\eps)T(\eps)^{n-1} = T(\alpha)T(\eps)^n,
\end{eqnarray*}
by induction and the preceding lemma.
\end{proof}

\begin{lemma}\label{dist}
Let $\alpha$ and $\beta$ be two irreducible characters of $K$. Then $0<d(\alpha,\;\beta)\leq m$ if and only if $m_K(\alpha,\; T^m(\beta))>0$.
\end{lemma}

\begin{proof}
This follows from Proposition~\ref{prop-entries2} since as noted before the matrix of the operator $T$ with respect to the basis of $C(K)$ given by the irreducible characters of $K$ is the matrix $\mathcal{S}$ defined in the previous section.\end{proof}

%By definition of $\sim$ one has that $\alpha \sim \beta$ if and only if $m_K(\alpha,\; T(\beta))>0$. Note that $\alpha$ is a constituent of $T(\alpha)$ for any irreducible character $\alpha$. Also if $\alpha$ is a constituent of $\beta$ then clearly $T(\alpha)$ is a constituent of $T(\beta)$. If $\alpha \sim \beta$ and $ \beta \sim \gamma $ then $\alpha$ is a constituent of $T(\beta)$ and $\beta$ is a constituent of $T(\gamma)$. Therefore $\alpha$ is a constituent of $T^2(\gamma)$.

\subsection{The operator $U$}
Define $U:C(H) \rightarrow C(H) $ such that $U(\chi)=\chi\downarrow_K^H\uparrow^H_K$. It follows from Eq.~(\ref{eq}) for $\alpha=\eps_H$ that $U^m(\chi)=\chi(\eps\uparrow_K^H)^m$ for all $m \geq 0$. Recall the equivalence relation $u^H_K$ on $\mathrm{Irr}(H)$ from Section~\ref{mult}. One has $\chi \sim \mu$ if and only if $\chi \downarrow_K^H$ and $\mu\downarrow_K^H$ have a common constituent. Then $u^H_K$ is the equivalence relation obtained by taking the transitive closure of $\sim$.
\begin{remark}\label{eqformula}\normalfont
Note that $\chi \sim \mu$ if and only if $m_{ _H}(\chi,\;U(\mu))>0$. Inductively it can be shown that $\chi\; u^H_K\; \mu$ if and only if there is $l >0$ such that $m(\chi,\;U^l(\mu))>0$.
\end{remark}

\begin{prop}\label{car}
A Hopf subalgebra $K$ is normal in $H$ if and only if $\eps_K$ by itself forms an equivalence class of $d^H_K$. \end{prop}

\begin{proof}
From the decomposition of $\mathcal{Z}(H)\cap K$ it follows that the integral element $\Lambda_{ _K}$ is central in $H$ \cite{ma}.
\end{proof}
\begin{remark}\label{grps}\normalfont
Suppose $H=\C G$ and $K=\C H$ for finite groups
$H \subset G$. Then $\mathcal{Z}(\C G)\cap \C H \subset\mathcal{Z}(\C N)$ where $N$ is the core of $H$ in $G$. This follows since a basis for $\mathcal{Z}(\C G)$ is given by $\sum_{g \in \mathcal{C}}g$ where $\mathcal{C}$ runs through all conjugacy classes of $G$.
\end{remark}

\section{Inclusion of group algebras} Let $H \subset G$ be an inclusion of finite groups and let $N:= \mathrm{core}_G(H)$
be the core of $H$ in $G$,  i.e., the largest normal subgroup in $G$ contained in $H$. We will use the short notations $u^G_H$ and $d^G_H$ for the equivalence relations $u^{\C G}_{\C H}$ and $d^{\C G}_{\C H}$
defined in subsections~\ref{you} and~\ref{dee}, respectively.

\begin{prop}\label{ker}
For all $n \geq 1$ one has that $\mathrm{ker}_G(U^n(\eps))=N$.
\end{prop}

\begin{proof}
Clearly $\mathrm{ker}_G(U^n(\eps))\subset \mathrm{ker}_G(U(\eps))$ since $U(\eps)$ has all the constituents inside $U^n(\eps)$. On the other hand $\mathrm{ker}_G(\chi^n) \supset \mathrm{ker}_G(\chi)$ for any character $\chi$. Thus $\mathrm{ker}_G(U^n(\eps))= \mathrm{ker}_G(U(\eps))$ for all $n \geq 1$. It remains to show $\mathrm{ker}_G(U(\eps))=N$. If $n \in N$ then one has that $n(g \ot_H 1)=g(g^{-1}ng)\ot_H1=g\ot_H1$ since $gng^{-1}\subset N\subset H$. Thus $N \subset \mathrm{ker}_G(U(\eps))$. On the other hand $xg \ot_H 1=g \ot_H 1$ implies that $xgH=gH$ and therefore $x \in gHg^{-1}$. Thus if $x \in \mathrm{ker}_G(U(\eps))$ then $x \in \cap_{g\in G}gHg^{-1}=N$.
\end{proof}

\begin{cor}\label{restr}
For all $n \geq 1$ one has that $\mathrm{ker}_H(T^n(\eps))=N$.
\end{cor}

\begin{proof}
It follows from the previous proposition since $U^n(\eps)\downarrow^G_H=T^n(\eps)$.
\end{proof}

\begin{cor}\label{epseqcls}
Let $H \subset G$ be a group inclusion and $N$ be the core of $H$ in $G$. Consider the equivalence relation $d_H^G$ on the irreducible characters of $H$ as above. Then the equivalence class of $\eps_H$ is $\mathrm{Irr}(H/N)$.
\end{cor}

\begin{proof}
If $\alpha \in \mathrm{Irr}(H/N)$ then $\alpha$ is a constituent of $\eps_N\uparrow^H_N$. Therefore $\alpha\uparrow^G_H$ is a constituent of $\eps_N\uparrow^G_N=U^m(\eps_H)$. But $U^m(\eps_H)\downarrow^G_H$ has all the constituents inside the equivalence class of $\eps_H$. Since $\alpha$ is a constituent of $T(\alpha)$ it follows that $\alpha \approx \eps$. Thus all the irreducible characters of $\mathrm{Irr}(H/N)$ are equivalent to $\eps_H$.

Conversely suppose that $\beta \sim \eps$. Then $\beta$ is a constituent of $T(\eps)$ and therefore by Corollary~\ref{restr} its restriction to $N$ contains $\beta(1)$ copies of the trivial character of $N$. Thus $\beta \in \mathrm{Irr}(H/N)$.
\end{proof}

Let $m$ be minimal such that $T^m(\eps)$ has all the possible constituents of all powers $T^n(\eps)$ with $n \geq 0$. Such $m$ exists since there are only finitely many characters that can appear and if they appear in a certain power of $T$ then they appear in any greater power of $T$.

\begin{remark} \label{normal}\normalfont
 Since $N \triangleleft G$ it is well known that $\chi \;u^G_N \;\mu$ if and only if
$\chi$ and $\mu$ have exactly the same irreducible constituents viewed as
$N$-characters by restriction. It also follows from \cite{Bker} that
$U_N(\chi)=\chi\eps\uparrow_N^G$ where $U_N:C(G)\rightarrow C(G)$ is given by
$U_N(\chi)=\chi\downarrow^G_N\uparrow^G_N$. For any irreducible character $\alpha \in
\mathrm{Irr}(N)$ the constituents of $\alpha\uparrow_N^G$ form an entire equivalence
class under $d^G_N$. Similarly, for any irreducible character $\chi \in
\mathrm{Irr}(G)$ the constituents of $\chi\downarrow_N^G$ form an entire equivalence
class under $d^G_N$. Thus, by Clifford theory, the equivalence classes of $
\mathrm{Irr}(N)$ under $d_N^G$ are just the $G$-orbits on $\mathrm{Irr}(N)$
(under the conjugation action).
\end{remark}

The following result from \cite{Bker} will be used in the sequel:

\begin{prop}\label{powerker}
Let $G$ be a group, $\chi$ a character of $G$ and $N=\mathrm{ker}(\chi)$. Then
$\eps_N^G$ has as irreducible constituents all the possible irreducible constituents
of all the powers of $\chi$.
\end{prop}

\begin{cor}\label{sameeqcls}
The equivalence relation $u^G_H$ is the same as the equivalence relation $u^G_N$ coming from $N\trianglelefteq G$. Thus the equivalence classes of $\mathrm{Irr}(G)$
under $u_H^G$ are in natural bijection with the $G$-orbits on $\mathrm{Irr}(N)$.
\end{cor}

\begin{proof}
Write as above $\chi \sim \mu$ if and only if $\chi\downarrow^G_H$ and $\mu\downarrow^G_H$ have a common constituent. If $\chi \sim \mu$ then clearly $\chi \;u^G_N \; \mu$. This implies that if $\chi \;u^G_H\; \mu$ then $\chi \;u^G_N \; \mu$. Conversely by Remark~\ref{eqformula} we see that $\chi \;u^G_N \; \mu$ if and only if $m_{ _G}(\chi, U_N(\mu))>0$. Remark~\ref{normal} implies that $U_N(\mu)=\mu\eps_N\uparrow_N^G$. On the other hand, using the previous proposition it follows that $m_{ _G}(\chi, \;U_N(\mu))=m_{ _G}(\chi, \mu\eps^G_N)= (\eps_N\uparrow_N^G,\;\mu\chi^*)>0$ if and only if $m_G(U^{m}(\eps),\;\chi\mu^*)=m_G(\chi, U^m(\mu))>0$.
\end{proof}

\begin{cor}
One has that the relation $\sim$ on $\mathrm{Irr}(G)$ coming from the inclusion $H\subset G$ is an equivalence relation if and only if $\eps\uparrow_{N}^G$ and $\eps\uparrow_H^G$ have the same constituents.
\end{cor}

\begin{proof}
$\sim$ is an equivalence relation if and only if $\eps\uparrow_H^G$ and $(\eps\uparrow_H^G)^m$ have the same constituents. But $(\eps_H\uparrow_H^G)^m$ has the same constituents as $\eps_N\uparrow_{N}^G$.
\end{proof}

\begin{prop}\label{factor}
Let $N \subset H \subset G$ with $N \lhd G$.
The depth of $H/N$ inside $G/N$ is less than or equal to the depth of $H$ in $G$. If $H$ has depth three or less in $G$ then $H/N$ has depth three or less in $G/N$.
\end{prop}

\begin{proof}
Let $\bar{T}:C(H/N) \rightarrow C(H/N)$ be
the operator $T$ defined as above but for the inclusion $H/N \subset G/N$. Since $Rep(G/N) \subset Rep(G)$ and $Rep(H/N) \subset Rep(H)$ it is easy to check that $\bar{T}$ is the restriction of $T$ to $C(H/N)$. Indeed both restriction and induction for the inclusion $H/N \subset G/N$ come from the restriction and induction for the inclusion $H \subset G$.

Then the proposition follows from Theorem~\ref{odddepth} and
Theorem~\ref{evendepth}.
\end{proof}

For example, with $G = S_4$, $H = D_8$ and
$$N = \{ (1), (12)(34), (13)(24), (14)(23)\} = V_4,$$
depth of $H < G$ is four (computed in Section~2),
while depth of $H/N \cong S_2 < G/N \cong S_3$ is three (computed graphically in Section~3).

Again let $G$ be a finite group, $H$ a proper subgroup of $G$, and $N := \mathrm{core}_G(H)$ denote the core of $H$ in $G$.  We say $\chi \in \mathrm{Irr}(G)$
and $\psi \in \mathrm{Irr}(H)$ are \textit{linked} if
\begin{equation}
0 \neq \bra \psi\uparrow^G| \chi \ket_G = \bra \psi| \chi\downarrow_H\ket_H.
\end{equation}
This defines a bipartite graph $\Gamma$ with vertices $\mathrm{Irr}(G) \cup \mathrm{Irr}(H)$ (the inclusion diagram
of the corresponding group algebras are a weighted
multigraph variant of this).
As usual, we denote by $\mathrm{Irr}(G|\kappa)$ the set
of all $\chi \in \mathrm{Irr}(G)$ such that $\langle \chi \downarrow_N, \kappa \rangle
\ne 0$, for $\kappa \in {\rm Irr}(N)$.

Proposition~\ref{sameeqcls} implies that the connected components of $\Gamma$ are
in bijection with the orbits of $G$ on $\mathrm{Irr}(N)$.

\subsection{A theorem with examples}

Recall that the \textit{core} $\mathrm{core}_G(H)$ of a subgroup $H < G$ is the largest normal subgroup of $G$ contained in $H$.  It is also defined by
$\mathrm{core}_G(H) = \bigcap_{x \in G} {^x}H$ where ${^x}H$ denotes
the subgroup $xHx^{-1}$ conjugate to $H$.

\begin{theorem}\label{intersection}
Let $H \subseteq G$ be an inclusion of finite groups, and suppose that $N:=\mathrm{core}_G(H)$ is the intersection of $m$
conjugates of $H$. Then $H$ has depth $\leq 2m$ in $G$. Moreover, if $N \subseteq
\mathcal{Z}(G)$ then $H$ has depth $\leq 2m-1$ in $G$.
\end{theorem}
\begin{proof}
Let $\alpha \in \mathrm{Irr}(H)$, and let $x \in G$. Then Mackey decomposition shows
that $\mathrm{Ind}_{H \cap xHx^{-1}}^H(\mathrm{Res}_{H \cap xHx^{-1}}^{xHx^{-1}} (
{^x\alpha}))$ is a summand of $T(\alpha) = \mathrm{Res}_H^G(\mathrm{Ind}_H^G(\alpha))$.
Thus $\mathrm{Res}_H^G(\mathrm{Ind}_{H \cap xHx^{-1}}^G(\mathrm{Res}_{H \cap xHx^{-1}}^{
xHx^{-1}}({^x\alpha})))$ is a summand of
$$T^2(\alpha) = \mathrm{Res}_H^G(\mathrm{Ind}_H^G(T(\alpha))).$$
Let $y \in G$. Then, by Mackey decomposition again,
$$\mathrm{Ind}_{H \cap yHy^{-1} \cap yxHx^{-1}y^{-1}}^H(\mathrm{Res}_{H \cap yHy^{-1}
\cap yxHx^{-1}y^{-1}}^{yxHx^{-1}y^{-1}}({^{yx}\alpha}))$$
is a summand of $T^2(\alpha)$. Continuing in this fashion, we see that, for $x_1 := 1,
x_2, \ldots, x_m \in G$,
$$\mathrm{Ind}_{x_1Hx_1^{-1} \cap \ldots \cap x_mHx_m^{-1}}^H(\mathrm{Res}_{
x_1Hx_1^{-1} \cap \ldots \cap x_mHx_m^{-1}}^{x_mHx_m^{-1}}({^{x_m}\alpha}))$$
is a summand of $T^{m-1}(\alpha)$. We can choose $x_1 = 1, x_2, \ldots, x_m =:z$ in
such a way that $x_1Hx_1^{-1} \cap \ldots \cap x_mHx_m^{-1} = N$. Then $T^{m-1}(\alpha)$
has a summand of the form
$$\mathrm{Ind}_N^H(\mathrm{Res}_N^{zHz^{-1}}({^z\alpha}))
= \mathrm{Ind}_N^H({^z}\mathrm{Res}_N^H(\alpha)).$$
Let $\beta$ be an irreducible constituent of $\mathrm{Res}_N^H(\alpha)$. Then
$\mathrm{Ind}_N^G({^z\beta}) = \mathrm{Ind}_N^G(\beta)$ is a summand of $\mathrm{Ind}_H^G
(T^{m-1}(\alpha))$. But the irreducible constituents of $\mathrm{Ind}_N^G(\beta)$ form a
complete equivalence class of $\mathrm{Irr}(G)$ under $u_H^G$, by
Corollary~\ref{sameeqcls}. Thus $\alpha$ has graphical distance at most $2m-1$ to
any $\chi \in \mathrm{Irr}(G)$. So $\alpha$ has graphical distance at most $2m-2$
to any set of irreducible constituents of $\mathrm{Res}_H^G(\chi)$, for any $\chi
\in \mathrm{Irr}(G)$, and the first part is proved.

Now suppose that $N \subseteq \mathcal{Z}(G)$. Then, in the notation above,
$\mathrm{Ind}_N^H({^z\beta}) = \mathrm{Ind}_N^H(\beta)$ is a summand of $T^{m-1}(
\alpha)$. But now the irreducible constituents of $\mathrm{Ind}_N^H(\beta)$ form a
complete equivalence class of $\mathrm{Irr}(H)$ under $d_H^G$. This shows that any
two irreducible characters of $H$ have graphical distance at most $2m-2$, so that
$H$ has depth at most $2m-1$ in $G$.
\end{proof}

We illustrate the theorem with three examples.

\begin{example}\normalfont
(1) Let $G = S_4$ and $H = D_8$, so that $N = V_4$ is the intersection of $m=2$
conjugates of $H$. By the theorem, $D_8$ has depth $\leq 4$ in $S_4$; indeed the depth is four by
our earlier computations. In the appendix, the depth of $D_{2n}$ in $S_n$ is
shown to be three for $n > 5$.

(2) Let $G=S_{n+1}$ and $H= S_n$ for some $n$. Then $N=1$, which is the intersection of
$m=n$ conjugates of $H$:
$$\{1\} = S_n \cap S_n^{(1 \, n+1)} \cap \cdots \cap S_n^{(n-1 \, n+1)}.$$
 By the theorem, $S_n$ has depth at most $2n-1$ in $S_{n+1}$.
We will see later that $2n-1$ is precisely the depth
of $S_n < S_{n+1}$.

(3) Let $G=A_6$ and $H=A_5$, so that $N=1$ again. A computation with character
tables shows that $A_5$ has depth $5$ in $A_6$. However, in this case, $N$ is
\textit{not} the intersection of $3$ conjugates of $H$, so the bound in the
theorem is not sharp here. The depth of the inclusion of alternating groups
$A_n \subseteq A_{n+1}$ will be computed in the appendix.
\end{example}

We obtain a corollary by recalling
that $G$ acts on the set of subgroups of $G$ by conjugation.
Let $N_G(H)$ be the normalizer of $H$ in $G$, which is the stabilizer  subgroup of $H$ under conjugation. The proof is a simple application
of the orbit counting theorem:

\begin{cor}
The depth of a subgroup $H$ of a finite group $G$ is bounded above by $2[G:N_G(H)]$.
\end{cor}

Since $N_G(H)$ contains each subgroup $K$ in which $H$ is normal, it follows that  a subnormal subgroup $H$ of subnormal depth
in $G$ (or defect) $r$ (cf. \cite{I}) has depth
less than or equal to $2m^{r-1}$, where $m$ is the maximal index of two consecutive subgroups in a subnormal series.

The following are examples of depth three or more subgroups from the literature on group theory.
\begin{example}\normalfont
Brodkey's theorem (cf. Theorem 1.37 in \cite[Isaacs]{I})
states that if a finite group $G$ has an abelian Sylow $p$-subgroup $H$, then the largest normal $p$-subgroup $\mathcal{O}_p(G) = N$ of $G$
is the intersection of two conjugates of $H$. In other terms then, $H$ is a depth
four or less subgroup in $G$; depth three or less if $N = \{ 1_G \}$.
\end{example}

\begin{example}\normalfont
If $G$ is $p$-solvable, where $p$ is odd and not a Mersenne prime, then the
largest normal $p$-subgroup $N$ of $G$ is an intersection of two Sylow
$p$-subgroups. If $p$ is even or a Mersenne prime, then $N$ is an intersection
of three Sylow $p$-subgroups \cite[Brewster-Hauck]{BH}. This in our terms
implies that the Sylow $p$-subgroup has depth $\leq$ $4$ or $6$,
respectively. If $N = 1$ then these numbers can be improved to $3$ and $5$,
respectively.
\end{example}

\begin{example}\normalfont
The theorem above implies that a subgroup $H$ of a finite group $G$ has at most depth
three if $H \cap xHx^{-1} = 1$ for some $x \in G$.
For example, a Sylow $p$-subgroup of $\mathrm{GL}(n,q)$
has depth three, as well as certain Borel and Weyl subgroups (for specific values of $n$ and $q = p^r$, left as an exercise to the interested reader)  \cite{AB}.
\end{example}

The results of this paper are suited for creating a program using GAP
to calculate the depth of  subgroups of suitably small groups.
We thank Susanne Danz for implementing such a program at the University of Jena.

In this paper we have found subgroups of  depth at  each odd
number (the symmetric group series), at depth four (the dihedral group in
$S_4$ with some additional examples) and a search with this program yields a subgroup of depth $6$
(the $108$-element normalizer subgroup of the Sylow $3$-subgroup of
the $432$-element affine group $AGL(3,2)$). We found
no subgroups
of depth an even number greater than $6$.

\begin{remark}\normalfont
Suppose $K < H < G$ is a tower of finite groups, where the subgroup $H < G$ is corefree and $m$ conjugates of $H$ have trivial intersection. Then the depth of the subgroup $K < G$ is bounded above by $2m-1$.
This follows from the same theorem since $K$ satisfies the same core
hypothesis.  For example, by the results of one of the examples above,
any subgroup $K$
of $S_n$ has depth less than or equal to $2n-1$ in $S_{n+1}$.
\end{remark}

\subsection{Computations for the operator $T$}
Suppose that $H$ is a subgroup of a finite group $G$. We denote by $\mathrm{Cl}(G)$ the
set of conjugacy classes of $G$ and by $\mathrm{CF}(G)$ the ring of complex class
functions on $G$. For a union $X$ of conjugacy classes of $G$, we denote by $\gamma_{G,X}$
the characteristic function of $X$ in $\mathrm{CF}(G)$. Then
$$\gamma_{G,X} \downarrow_H = \gamma_{H,H \cap X}.$$
Similarly, if $C$ and $D$ denote the conjugacy classes in $G$ and $H$, respectively, of
an element in $H$ then an easy computation shows that
$$\gamma_{H,D} \uparrow^G = \frac{|G|}{|H|} \cdot \frac{|D|}{|C|} \gamma_{G,C}.$$
This implies that the eigenvectors of the linear map
$$T: \mathrm{CF}(H) \longrightarrow \mathrm{CF}(H), \quad \chi \longmapsto
\chi \uparrow^G \downarrow_H,$$
corresponding to nonzero eigenvalues are precisely the class functions $\gamma_{H,C \cap
H}$ ($C \in \mathrm{Cl}(G)$, $C \cap H \ne \emptyset$). Moreover, the eigenvalue of $T$
corresponding to an eigenvector $\gamma_{C, C \cap H}$ is clearly
$$\frac{|G|}{|H|} \cdot \frac{|C \cap H|}{|C|}.$$
We denote by
$$t := |\{\frac{|G|}{|H|} \cdot \frac{|C \cap H|}{|C|}: C \in \mathrm{Cl}(G), \;
C \cap H \ne \emptyset \}|$$
the number of distinct nonzero eigenvalues of $T$. Then the minimum polynomial of $T$
has degree $t$ or $t+1$. Since $\mathcal{S}$
is the matrix of $T$ with respect to the basis $
\mathrm{Irr}(H)$ of $\mathrm{CF}(H)$, we get an equation
$$ \mathcal{S}^{t+1} + a_1\mathcal{S}^t + \cdots + a_t\mathcal{S} + a_{t+1} = 0$$
where $a_1, \ldots, a_{t+1} \in \C$. Thus $A(\mathcal{S}^{t+1}) \subseteq
A(\mathcal{S}^t)$, so that $H$ has
depth $\leq 2t+1$ in $G$.

We also note that all eigenvalues of $T$ are nonzero if and only if $T$ is surjective.
This is equivalent to the condition that two elements in $H$ are conjugate in $G$ if
and only if they are already conjugate in $H$. In this case, the minimum polynomial of
$T$ has degree $t$. So, arguing as above, we conclude that $H$ has depth $\leq 2t-1$ in $G$.
We summarize:

\begin{theorem}
(i) The nonzero eigenvalues of $\mathcal{S}$ are the numbers
$$\frac{|G|}{|H|} \cdot \frac{|C \cap H|}{|C|} \quad \quad\quad (C \in \mathrm{Cl}
(G), \; C \cap H \ne \emptyset).$$

(ii) The subgroup $H$ of $G$ has depth $\leq 2t+1$ in $G$ where $t$ denotes the number of
distinct nonzero eigenvalues of $\mathcal{S}$.

(iii) All eigenvalues of $\mathcal{S}$ are nonzero if and only if any two elements in
$H$ which are conjugate in $G$ are already conjugate in $H$. In this case, $H$ has
depth $\leq 2t-1$ in $G$.
\end{theorem}

\begin{example}\normalfont
For the inclusion of the alternating groups $A_4 < A_5$ may be
checked that the minimum polynomial of $\mathcal{S}$ is
$X(X-1)(X-2)(X-5)$. By the theorem above, $A_4$ has depth $\leq 7$
in $A_5$. Computing the powers of $M$, one sees that the subgroup
$A_4 < A_5$ has depth five. The depth of the inclusion $A_n
\subseteq A_{n+1}$ for arbitrary $n$ will be computed in the
appendix.
\end{example}

\begin{example}\normalfont
Consider the inclusion $S_3 < S_4$ of permutation groups.  It can be
computed that the minimal polynomial $\mathcal{S}$ in this case is
given by $m(X) = X^3 - 7X^2 +14 X -8 $ $ = (X-4)(X-2)(X-1)$. The
nonzero eigenvalues of $\mathcal{S}$ are $1,2,4$. By the theorem
above, the depth of $S_3 < S_4$ is at most five, which is the
precise depth of the extension as we will see in the next
subsection.
\end{example}

\subsection{Depth of inclusions of symmetric groups}
In this subsection we will prove the following:

\begin{theorem}
\label{th-sym}
For any $n\geq 2$ the standard inclusion $S_n \subset S_{n+1}$ has  depth $2n-1$.
\end{theorem}

In order to prove the theorem, we recall that the irreducible characters of $S_n$
are in bijection with partitions of $n$. Moreover, partitions of $n$ can be
visualized by their Young diagrams. For example, the trivial character of $S_n$
corresponds to the trivial partition $(n)$ of $n$, and the Young diagram of $(n)$
is a row of $n$ boxes. Similarly, the sign character of $S_n$ corresponds to the
partition $(1^n)= (1,\ldots,1)$,
and the Young diagram of $(1^n)$ is a column of $n$ boxes.

By the branching rules, restricting an irreducible character of $S_{n+1}$ to $S_n$
means removing a box from the corresponding Young diagram, and inducing an
irreducible character of $S_n$ to $S_{n+1}$ means adding a box to the corresponding
Young diagram.

By Theorem~\ref{intersection} above, the inclusion $S_n \subseteq S_{n+1}$ has
depth $\leq 2n-1$. It is easy to give an alternative proof of this, based on the
combinatorics of Young diagrams. These ideas are explained in more detail in the
appendix where they are also used to determine the  depth of the inclusion
of alternating groups $A_n \subseteq A_{n+1}$.

It only remains to show that the inclusion matrix of $S_n \subseteq S_{n+1}$ does not satisfy a depth
$2n-2$ inequality. For this we may argue as follows:

The sign character of $S_{n+1}$, denoted by $V_u$, restricts to the sign character
$\sigma$ of $S_n$. Thus, in
the notation of Section~3, the set $\mathcal{V}_u$ consists of $\sigma$
alone. It is easy to see that $d(\eps,\sigma) = n-1$:
\medskip
\begin{center}
\begin{minipage}[c]{20mm}
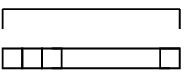\end{minipage}$\quad \longmapsto \quad$\begin{minipage}[c]{20mm}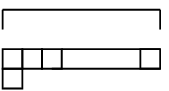\end{minipage}
$\quad\longmapsto\quad $\begin{minipage}[c]{20mm}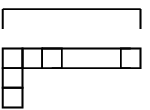\end{minipage}
\end{center}
\begin{center}
$\longmapsto\;\cdots\;\longmapsto \quad$ \begin{minipage}[c]{20mm}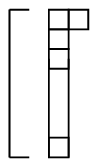\end{minipage}$\;\longmapsto\quad$\begin{minipage}[c]{20mm}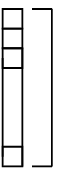\end{minipage}
\end{center}
\bigskip
It follows that $m(V_u) = n-1$. Thus Theorem~\ref{evendepth} shows that the inclusion matrix of $S_n$ in $S_{n+1}$ cannot
satisfy a depth $2n-2$ inequality.

This result also applies to the semisimple Hecke algebras:
$H(q,n)$ is depth $2n - 1$ in $H(q,n+1)$, since they share
the same representation theory with the permutation groups
$S_n < S_{n+1}$ (see \cite{GHJ}).

%%%%%%%%%%%%%%%%%%%%%%%%%%%%%%%%%%%%%%%%%%%%%%%%%%%%%%%%%%%%%%%%%%%%%%%%%%%%%%%
\bigskip\bigskip

\begin{appendix}
\begin{center}
{\bf \Large Appendix: Depth of subgroups -- some examples}\\

\medskip
\textit{Susanne Danz\footnote{Current address: Mathematical Institute,
University of Oxford, St. Giles' 24-29, OX1 3LB, Oxford, UK, {\tt danz@maths.ox.ac.uk}\\
Supported by DFG Grant No. DA 1115-1/1}
 and Burkhard K\"ulshammer}\\
\smallskip
\textit{Mathematisches Institut, Friedrich-Schiller-Universit\"at,}\\
\textit{07737 Jena, Germany}
\end{center}

\section{Inclusions of symmetric and alternating groups}
Throughout this section, let $n\geq 1$, let $\mathfrak{S}_n$ denote
the symmetric group of degree $n$, and let $\mathfrak{A}_n$ denote the
alternating group of degree $n$. Moreover, let $\mathcal{P}_n$ be the
set of all partitions of $n$. By Theorem~\ref{th-sym}, we
know that for $n\geq 2$
the ring extension $\mathbb{C}\mathfrak{S}_n\subset \mathbb{C}\mathfrak{S}_{n+1}$
is of depth $2n-1$.
We now aim to determine the depth of the
ring extension $\mathbb{C}\mathfrak{A}_n\subset \mathbb{C}\mathfrak{A}_{n+1}$.
Moreover, we will give a combinatorial proof of Theorem~\ref{th-sym}.
Before stating the results, we fix some further notation.

\begin{remark}\label{rem:alt}
\normalfont
(a) For $\lambda\in\mathcal{P}_n$, we denote
the conjugate partition by $\lambda'$. That is, the Young diagram of $\lambda'$ is
obtained by transposing the Young diagram of $\lambda$.
For $\lambda \in\mathcal{P}_n$, let $\chi^{\lambda}$ be the corresponding ordinary irreducible
$\mathfrak{S}_n$-character.
If $\lambda= \lambda'$ then $\chi^{\lambda}\downarrow_{\mathfrak{A}_n}=\chi^{\lambda}_++\chi^{\lambda}_-$,
for irreducible $\mathfrak{A}_n$-characters $\chi^{\lambda}_+\neq \chi^{\lambda}_-$. We
choose our labelling in accordance with \cite{JK}, Sec. 2.5.
With this convention, for $\alpha'=\alpha\in\mathcal{P}_{n+1}$ and
$\lambda'=\lambda\in\mathcal{P}_n$ such that
$\langle\chi^{\alpha}\downarrow_{\mathfrak{S}_n},\chi^{\lambda}\rangle\neq 0$, we have
$\langle\chi^{\alpha}_+\downarrow_{\mathfrak{A}_n},\chi^{\lambda}_+\rangle\neq 0=
\langle\chi^{\alpha}_+\downarrow_{\mathfrak{A}_n},\chi^{\lambda}_-\rangle$
and $\langle\chi^{\alpha}_-\downarrow_{\mathfrak{A}_n},\chi^{\lambda}_-\rangle\neq 0=
\langle\chi^{\alpha}_-\downarrow_{\mathfrak{A}_n},\chi^{\lambda}_+\rangle$
(see \cite{JK}, Thm. 2.5.13, and \cite{BO}).
If $\lambda\neq \lambda'$ then
$\chi^{\lambda}\downarrow_{\mathfrak{A}_n}=\chi^{\lambda'}\downarrow_{\mathfrak{A}_n}$ is irreducible.
We may then suppose that $\lambda> \lambda'$, and set
$\chi^{\lambda}_0:=\chi^{\lambda}\downarrow_{\mathfrak{A}_n}$.
Here ``$\geq$'' denotes the usual lexicographic ordering on partitions.\\

(b) We consider the bipartite graphs $\Gamma(\mathfrak{S}_n)$ and
$\Gamma(\mathfrak{A}_n)$. Here $\Gamma(\mathfrak{S}_n)$ has vertices
$V:=\mathcal{P}_n\cup\mathcal{P}_{n+1}$ and edges
$$E:=\{(\alpha,\lambda)\in \mathcal{P}_{n+1}\times\mathcal{P}_n\mid
\langle\chi^{\alpha}\downarrow_{\mathfrak{S}_n},\chi^{\lambda}\rangle\neq 0\}.$$
The graph $\Gamma(\mathfrak{A}_n)$ has vertices $\widetilde{V}:=V(n)\cup V(n+1)$
and edges $\widetilde{E}$ where
\begin{align*}
V(n)&:=\{[\lambda,0]\mid \lambda\in\mathcal{P}_n,\, \lambda>\lambda'\}\cup\{[\lambda,+],[\lambda,-]\mid \lambda=\lambda'\in\mathcal{P}_n\},\\
V(n+1)&:=\{[\alpha,0]\mid \alpha\in\mathcal{P}_{n+1},\, \alpha>\alpha'\}\cup\{[\alpha,+],[\alpha,-]\mid \alpha=\alpha'\in\mathcal{P}_{n+1}\},\\
\widetilde{E}&:=\{([\alpha,x],[\lambda,y])\in V(n+1)\times V(n)\mid \langle\chi^{\alpha}_x\downarrow_{\mathfrak{A}_n},\chi^{\lambda}_y\rangle\neq 0\}.
\end{align*}

(c) Let $\lambda,\mu\in\mathcal{P}_n$ with corresponding Young diagrams
$[\lambda]$ and $[\mu]$, respectively. We set
$$d(\lambda,\mu):=|[\lambda]\setminus [\mu]|+|[\mu]\setminus [\lambda]|=2(n-|[\lambda]\cap[\mu]|).$$
With this notation, we have:
\end{remark}

\begin{prop}\label{prop:depthS_n}
Let $n\geq 2$, and let $\lambda,\, \mu\in\mathcal{P}_n$.
Then $d(\lambda,\mu)$ is the length of a shortest path from $\lambda$ to
$\mu$ in $\Gamma(\mathfrak{S}_n)$. In particular, the ring extension
$\mathbb{C}\mathfrak{S}_n\subset\mathbb{C}\mathfrak{S}_{n+1}$
has depth $2n-1$.
\end{prop}

\begin{proof}
Let $\lambda,\,\mu\in\mathcal{P}_n$, and set $2m:=d(\lambda,\mu)$.
We argue with induction on $m$, in order to show that in $\Gamma(\mathfrak{S}_n)$ there
is a path of length $2m$ from $\lambda$ to $\mu$. For $m=0$
this is trivially true, so we may now suppose that $m\geq 1$.
We construct a partition $\lambda^1$ of $n$ such that $d(\lambda,\lambda^1)=2$,
$d(\lambda^1,\mu)=2m-2$, and such that there is a path of length $2$ from $\lambda$ to
$\lambda^1$. Since $\lambda\neq \mu$, we have
$[\lambda]\not\subseteq [\mu]\not\subseteq [\lambda]$. Thus there is some $i\in\mathbb{N}$
such that
$(i,\lambda_i)\in [\lambda]\setminus [\mu]$ and $(i+1,\lambda_i)\notin [\lambda]$.
That is, $(i,\lambda_i)$ is a removable node of $[\lambda]$. Analogously, there are
some $r,s\in\mathbb{N}$ such that $(r,s)\in [\mu]\setminus [\lambda]$.
We may suppose further that $(t,s)\in [\lambda]$, for $1\leq t\leq r-1$, and $(r,u)\in [\lambda]$,
for $1\leq u\leq s-1$. So $(r,s)$ is an addable node of $[\lambda]$.
We define $\alpha\in\mathcal{P}_{n+1}$ with Young diagram $[\alpha]:=[\lambda]\cup\{(r,s)\}$.
Assume that $(i,\lambda_i)$ is not a removable node of $[\alpha]$. This can happen only if
$(r,s)=(i+1,\lambda_i)$ or $(r,s)=(i,\lambda_i+1)$. But, since $(r,s)\in [\mu]$, this
implies also $(i,\lambda_i)\in [\mu]$, a contradiction. Therefore,
$[\lambda^1]:=[\alpha]\setminus \{(i,\lambda_i)\}$ is the Young diagram
of a partition $\lambda^1\in\mathcal{P}_n$ with $d(\lambda,\lambda^1)=2$,
$d(\lambda^1,\mu)=2m-2$, and
$\langle\chi^{\alpha}\downarrow_{\mathfrak{S}_n},\chi^{\lambda}\rangle\neq 0\neq
\langle\chi^{\alpha}\downarrow_{\mathfrak{S}_n},\chi^{\lambda^1}\rangle$.
So there is a path of length 2 from $\lambda$ to $\lambda^1$ in $\Gamma(\mathfrak{S}_n)$.
By induction, there is a path of length $2m-2$ from $\lambda^1$ to
$\mu$ in $\Gamma(\mathfrak{S}_n)$. So we obtain a path
$$\begin{xy}\xymatrix{\alpha^1\ar@{-}[d] \ar@{-}[rd]& \alpha^2\ar@{-}[d] \ar@{-}[rd]&&\cdots & \alpha^m\ar@{-}[d] \ar@{-}[rd]\\\lambda=\lambda^0 & \lambda^1 & \lambda^2 &\cdots &\lambda^{m-1} &\lambda^m=\mu}
\end{xy}$$
of length $2m$ from $\lambda$ to $\mu$ in $\Gamma(\mathfrak{S}_n)$. Conversely, let
$$\begin{xy}\xymatrix{\beta^1\ar@{-}[d] \ar@{-}[rd]& \beta^2\ar@{-}[d] \ar@{-}[rd]&&\cdots & \beta^r\ar@{-}[d] \ar@{-}[rd]\\
\lambda=\mu^0 & \mu^1 & \mu^2 &\cdots &\mu^{r-1} &\mu^r=\mu}
\end{xy}$$
be a shortest path in $\Gamma(\mathfrak{S}_n)$. That is, $2r\leq d(\lambda,\mu)$.
We argue with induction on $r$ to show
$d(\lambda,\mu)=2r$. If $r=0$ then $\lambda=\mu$, and $d(\lambda,\mu)=0=2r$. Next suppose
that $r=1$. Then $\lambda\neq \mu$, and $[\mu]$ is obtained by first adding a node
$(i,j)$ to $[\lambda]$, and then removing a node $(r,s)\neq (i,j)$ from
$[\lambda]\cup\{(i,j)\}=[\beta^1]$. That is $d(\lambda,\mu)=2$. Now we may suppose that
$r\geq 2$. By induction, $d(\lambda,\mu^{r-1})=2(r-1)$ and $d(\mu^{r-1},\mu)=2$. So
$d(\lambda,\mu)\leq 2r$, and thus $2r=d(\lambda,\mu)$. This proves the first part
of the statement. \\

Since, for any $\lambda,\mu\in\mathcal{P}_n$, we have $(1,1)\in [\lambda]\cap [\mu]$,
it follows that $d(\lambda,\mu)\leq 2(n-1)$. Moreover, for $\lambda:=(n)$ and
$\mu:=(1^n)$, we have $d(\lambda,\mu)=2(n-1)$, so that, by what we have just shown
and  Theorem~\ref{odddepth}, the extension
$\mathbb{C}\mathfrak{S}_n\subset\mathbb{C}\mathfrak{S}_{n+1}$ has depth $2(n-1)+1$.
On the other hand, $\chi^{(n)}=\chi^{(n+1)}\downarrow_{\mathfrak{S}_n}$. Therefore, for $n>2$,
Theorem~\ref{evendepth} implies that the inclusion matrix of
$\mathbb{C}\mathfrak{S}_n\subset\mathbb{C}\mathfrak{S}_{n+1}$ cannot satisfy a depth $2(n-1)$ inequality.
If $n=2$ then $\mathbb{C}\mathfrak{S}_n\subset\mathbb{C}\mathfrak{S}_{n+1}$
cannot have depth $2(n-1)=2$, since $\mathfrak{S}_2\not\trianglelefteq\mathfrak{S}_3$.
This completes the proof of the proposition.
\end{proof}

\begin{lemma}\label{lemma:order}
Let $\lambda,\mu\in\mathcal{P}_n$.
Suppose that $\mu'=\mu\neq \lambda=\lambda'$. Then
$d(\lambda,\mu)\geq 4$.
\end{lemma}

\begin{proof}
Let $\mu'=\mu\neq \lambda=\lambda'$. Then there are some
$i,j\in\{1,\ldots,n\}$ such that
$(i,\lambda_i)\notin [\mu]$ and $(j,\mu_j)\notin [\lambda]$.
Since both $\lambda$ and $\mu$ are symmetric, also
$(\lambda_i,i)\in [\lambda]\setminus [\mu]$ and
$(\mu_j,j)\in [\mu]\setminus [\lambda]$. In particular, $i\neq \lambda_i$ or $j\neq \mu_j$,
and also $i\neq j$. Hence $d(\lambda,\mu)\geq 3$, and
so $d(\lambda,\mu)\geq 4$, since $d(\lambda,\mu)$ is even.
\end{proof}

\begin{example}\normalfont
Consider, for instance, the symmetric partitions $\lambda=(4,3,2,1)$ and
$\mu=(5,2,1^3)$ of $10$. Then we have $d(\lambda,\mu)=2(10-4-2-1-1)=4$.
\end{example}

\begin{prop}\label{prop:depthA_n}
Let $n\geq 3$. Then the ring extension
$\mathbb{C}\mathfrak{A}_n\subset\mathbb{C}\mathfrak{A}_{n+1}$ has depth
$2(n-\lceil\sqrt{n}\rceil)+1$.
\end{prop}

\begin{proof}
In consequence of Theorem~\ref{odddepth} and Theorem~\ref{evendepth}, it suffices to
show the following:

\begin{enumerate}
\item[(1)] For any $v,w\in V(n)$, there is a path of
length at most $2(n-\lceil\sqrt{n}\rceil)$ from $v$ to $w$
in $\Gamma(\mathfrak{A}_n)$, and
\item[(2)] there is some $v\in V(n)$ such that in $\Gamma(\mathfrak{A}_n)$
there is no path of length less than $2(n-\lceil\sqrt{n}\rceil)$
from $v$ to $[(n),0]$.
\end{enumerate}
For this, let $v,w\in V(n)$. Suppose first that $\lambda=\lambda'\in\mathcal{P}_n$
and that $v:=[\lambda,+]$ and $w:=[\lambda,-]$. Let further $\alpha\in\mathcal{P}_{n+1}$
with Young diagram $[\alpha]:=[\lambda]\cup\{(1,\lambda_1+1)\}$. Then $\alpha>\alpha'$,
and $\langle\chi^{\alpha}_0\downarrow_{\mathfrak{A}_n},\chi^{\lambda}_+\rangle=1=
\langle\chi^{\alpha}_0\downarrow_{\mathfrak{A}_n},\chi^{\lambda}_-\rangle$,
by \cite{BO}. Hence in $\Gamma(\mathfrak{A}_n)$ there is a path of length
$2\leq 2(n-\lceil\sqrt{n}\rceil)$ from $v$ to $w$.\\

Therefore, from now on, we may suppose that $v=[\lambda,x]$ and
$w=[\mu,y]$, for some $\lambda\geq \lambda'$ and $\mu\geq\mu'$ with
$\lambda\neq \mu$, and appropriate $x,y\in\{0,+,-\}$. We set $2m:=d(\lambda,\mu)$, and
show that there is a path from $v$ to $w$ in $\Gamma(\mathfrak{A}_n)$ of length $2m$.
Note that, since $\lambda\geq \lambda'$ and $\mu\geq \mu'$, we
must have $\lambda_1\geq \lceil\sqrt{n}\rceil$ and also
$\mu_1\geq \lceil\sqrt{n}\rceil$. So $2m\leq 2(n-\lceil\sqrt{n}\rceil)$, and
we then get (1).\\

First of all, by Proposition \ref{prop:depthS_n}, there is a path
$$\begin{xy}\xymatrix{\alpha^1\ar@{-}[d] \ar@{-}[rd]& \alpha^2\ar@{-}[d] \ar@{-}[rd]&&\cdots & \alpha^m\ar@{-}[d] \ar@{-}[rd]\\\lambda=\lambda^0 & \lambda^1 & \lambda ^2 &\cdots &\lambda^{m-1} &\lambda^m=\mu}
\end{xy}$$
of length $2m$ in $\Gamma(\mathfrak{S}_n)$. Here $\lambda^0,\ldots,\lambda^m\in\mathcal{P}_n$,
and $\alpha^1,\ldots,\alpha^m\in\mathcal{P}_{n+1}$. We now construct a path
$$\begin{xy}\xymatrix{[\tilde{\alpha}^1,z_1]\ar@{-}[d] \ar@{-}[rd]&&\cdots & [\tilde{\alpha}^m,z_m]\ar@{-}[d] \ar@{-}[rd]\\ \mathop{[\lambda,x]}\limits_{\displaystyle =[\tilde{\lambda}^0,x_0]} & [\tilde{\lambda}^1,x_1] &\cdots &[\tilde{\lambda}^{m-1},x_{m-1}] &\mathop{[\tilde{\lambda}^m,x_m]}\limits_{\displaystyle =[\mu,y]}}
\end{xy}$$
in $\Gamma(\mathfrak{A}_n)$ as follows. For $i=1,\ldots,m$, we set
$\tilde{\lambda}^i:=\max\{\lambda^i,(\lambda^i)'\}$ and
$\tilde{\alpha}^i:=\max\{\alpha^i,(\alpha^i)'\}$
where the maxima are taken with respect to the lexicographic ordering on partitions.
We then determine the ``signs'' $x_0,\ldots, x_m, z_1,\ldots, z_m$ inductively.
Of course, $x_0=x$. So we may suppose that $i\geq 1$ and that we have already determined
$x_0,\ldots, x_{i-1}$ and $z_1,\ldots, z_{i-1}$. In order to fix $z_i$ and $x_i$, we
distinguish four cases.\\

\textit{Case 1.} $\tilde{\lambda}^{i-1}\neq (\tilde{\lambda}^{i-1})'$ and
$\tilde{\alpha}^i\neq (\tilde{\alpha}^i)'$. In this case we set $z_i:=0$ and
$$x_i:=\begin{cases}
+, &\text{ if } \tilde{\lambda}^i = (\tilde{\lambda}^i)'\text{ and } y\in\{0,+\},\\
-, &\text{ if } \tilde{\lambda}^i = (\tilde{\lambda}^i)'\text{ and } y=-,\\
0, &\text{ if } \tilde{\lambda}^i \neq (\tilde{\lambda}^i)'.
\end{cases}$$

\textit{Case 2.} $\tilde{\lambda}^{i-1}\neq (\tilde{\lambda}^{i-1})'$ and
$\tilde{\alpha}^i= (\tilde{\alpha}^i)'$. If $\tilde{\lambda}^i\neq(\tilde{\lambda}^i)'$ then
we set $z_i:=+$ and $x_i:=0$, otherwise
$$z_i:=x_i:=\begin{cases}
+, &\text{ if } y\in\{0,+\},\\
-, &\text{ if } y=-.
\end{cases}$$

\textit{Case 3.} $\tilde{\lambda}^{i-1}= (\tilde{\lambda}^{i-1})'$ and
$\tilde{\alpha}^i\neq (\tilde{\alpha}^i)'$. We then set $z_i:=0$ and
$$x_i:=\begin{cases}
+, &\text{ if } \tilde{\lambda}^i = (\tilde{\lambda}^i)'\text{ and } y\in\{0,+\},\\
-, &\text{ if } \tilde{\lambda}^i = (\tilde{\lambda}^i)'\text{ and } y=-,\\
0, &\text{ if } \tilde{\lambda}^i \neq (\tilde{\lambda}^i)'.
\end{cases}$$

\textit{Case 4.} $\tilde{\lambda}^{i-1}= (\tilde{\lambda}^{i-1})'$ and
$\tilde{\alpha}^i=(\tilde{\alpha}^i)'$. Then, in particular,
$\tilde{\lambda}^i\neq (\tilde{\lambda}^i)'$,
by Lemma \ref{lemma:order}. Thus we may set $z_i:=x_{i-1}$ and $x_i:=0$.\\

Note that this construction ensures that $x_m=y$. Moreover, by \cite{BO}, for
$i=1,\ldots, m$, we get a path of length $2$ from $[\tilde{\lambda}^{i-1},x_{i-1}]$
to $[\tilde{\lambda}^i,x_i]$ in $\Gamma(\mathfrak{A}_n)$, hence a path of length $2m$
from $[\lambda,x]$ to $[\mu,y]$. Note further that these considerations also show the
following: in the case where $\mu=\mu'$ there are both a path of length $2m$ from
$[\lambda,x]$ to $[\mu,+]$ and a path of length $2m$ from $[\lambda,x]$
to $[\mu,-]$.\\

In order to prove (2), let conversely $\lambda,\mu\in V(n)$ such that
$2\leq d(\lambda,\mu)=:2m$, let $x,y\in\{0,+,-\}$ be appropriate signs, and
let
$$\begin{xy}\xymatrix{[\beta^1,z_1]\ar@{-}[d] \ar@{-}[rd]&&\cdots & [\beta^r,z_r]\ar@{-}[d] \ar@{-}[rd]\\
\mathop{[\lambda,x]}\limits_{\displaystyle =[\lambda^0,x_0]} & [\lambda^1,x_1] &\cdots &[\lambda^{r-1},x_{r-1}] &\mathop{[\lambda^r,x_r]}\limits_{\displaystyle =[\mu,y]}}
\end{xy}$$
be a shortest path from $[\lambda,x]$ to $[\mu,y]$ in $\Gamma(\mathfrak{A}_n)$.
Then $r\leq m$, by what we have shown above. We also observe that the partitions
$\lambda^0,\lambda^1,\ldots,\lambda^r$ must be pairwise different. To see this,
assume that $\lambda^i=\lambda^j$, for some $0\leq i<j\leq r$. Then we may suppose that
$x_i=+$ and $x_j=-$. As we have seen above, there is a path of length 2 from $[\lambda^i,x_i]$ to
$[\lambda^j,x_j]$ so that $j=i+1$, since the given path is as short as possible.
If $j<r$ then $\lambda^{i+2}\neq\lambda^i=\lambda^{i+1}$, by the minimality of
$r$. But, since there is a path of length $2$ from $[\lambda^{i+1},x_{i+1}]=[\lambda^i,x_{i+1}]$ to
$[\lambda^{i+2},x_{i+2}]$, there is also a path of length 2 from $[\lambda^i,x_i]$
to $[\lambda^{i+2},x_{i+2}]$, as we have proved above. But this contradicts the minimality
of $r$. If $i+1=r$ then $i>0$ and $\lambda^{i-1}\neq \lambda^i=\lambda^{i+1}$.
This implies that there is a path of length 2 from $[\lambda^{i-1},x_{i-1}]$ to
$[\lambda^{i+1},x_{i+1}]$, which is again a contradiction to the minimality of $r$.\\

We now set $\tilde{\lambda}^0:=\lambda^0=\lambda$.
Since $\langle\chi^{\beta^1}_{z_1}\downarrow_{\mathfrak{A}_n},\chi^{\lambda}_x\rangle\neq 0
\neq \langle\chi^{\beta^1}_{z_1}\downarrow_{\mathfrak{A}_n},\chi^{\lambda^1}_{x_1}\rangle$, there is some
$\tilde{\beta}^1\in\{\beta^1,(\beta^1)'\}$ such that
$\langle\chi^{\tilde{\beta}^1}\downarrow_{\mathfrak{S}_n},\chi^{\lambda}\rangle\neq 0$.
Having fixed $\tilde{\beta}^1$, we can find
$\tilde{\lambda}^1\in\{\lambda^1,(\lambda^1)'\}$
such that also $\langle\chi^{\tilde{\beta}^1}\downarrow_{\mathfrak{S}_n},\chi^{\tilde{\lambda}^1}\rangle\neq 0$.
Let now $i\geq 2$. We may argue inductively, and suppose that we have a path
$$\begin{xy}\xymatrix{\tilde{\beta}^1\ar@{-}[d] \ar@{-}[rd]& \tilde{\beta}^2\ar@{-}[d] \ar@{-}[rd]&&\cdots & \tilde{\beta}^{i-1}\ar@{-}[d] \ar@{-}[rd]\\
\lambda=\tilde{\lambda}^0 & \tilde{\lambda}^1 & \tilde{\lambda}^2 &\cdots &\tilde{\lambda}^{i-2} &\tilde{\lambda}^{i-1}}
\end{xy}$$
in $\Gamma(\mathfrak{S}_n)$, for appropriate $\tilde{\beta}^j\in\{\beta^j,(\beta^j)'\}$,
$\tilde{\lambda}^j\in\{\lambda^j,(\lambda^j)'\}$, and all $j=1,\ldots, i-1$.
We then choose $\tilde{\beta}^i\in \{\beta^i,(\beta^i)'\}$ and
$\tilde{\lambda}^i\in\{\lambda^i,(\lambda^i)'\}$ such that
$\langle\chi^{\tilde{\beta}^i}\downarrow_{\mathfrak{S}_n},\chi^{\tilde{\lambda}^{i-1}}\rangle\neq 0
\neq\langle\chi^{\tilde{\beta}^i}\downarrow_{\mathfrak{S}_n},\chi^{\tilde{\lambda}^i}\rangle$.
In this way we obtain in $\Gamma(\mathfrak{S}_n)$ a path of length $2r$ from
$\lambda$ to $\mu$, or from $\lambda$ to $\mu'$.\\

Now let
$\lambda:=(n)$, and let $\mu\in\mathcal{P}_n$ be such that
\begin{align*}
\{1,\ldots,\lfloor\sqrt{n}\rfloor\}\times \{1,\ldots,\lfloor\sqrt{n}\rfloor\}&\subseteq [\mu]\cap [\mu'],\\
[\mu]\cup [\mu']&\subseteq \{1,\ldots,\lceil\sqrt{n}\rceil\}\times \{1,\ldots,\lceil\sqrt{n}\rceil\},
\end{align*}
and such that $\mu\geq \mu'$. Then $2\leq d(\lambda,\mu)=2(n-\lceil\sqrt{n}\rceil)$, and
$2(n-\lceil\sqrt{n}\rceil)\leq d(\lambda,\mu')\leq 2(n-\lceil\sqrt{n}\rceil)+2$.
Assume that in $\Gamma(\mathfrak{A}_n)$ there is a path of length $2r<2(n-\lceil\sqrt{n}\rceil )$
from $[(n),0]$ to $[\mu,y]$, for some
admissible $y\in\{0,+,-\}$. Then, by the above considerations, in $\Gamma(\mathfrak{S}_n)$
there is a path of length $2r$ from $(n)$ to $\mu$, or from $(n)$ to $\mu'$.
But, since $2r<\min\{d((n),\mu),d((n),\mu')\}$, this is impossible, by Proposition
\ref{prop:depthS_n}. Therefore, we have now also shown (2), and the
assertion of the proposition follows.
\end{proof}

\begin{remark}\normalfont
Note that the ring extensions $\mathbb{C}\mathfrak{S}_1\subset\mathbb{C}\mathfrak{S}_2$ and
$\mathbb{C}\mathfrak{A}_2\subset\mathbb{C}\mathfrak{A}_3$ are clearly of depth 2, since
$\mathfrak{S}_1\unlhd \mathfrak{S}_2$ and $\mathfrak{A}_2\unlhd \mathfrak{A}_3$.
\end{remark}

%%%%%%%%%%%%%%%%%%%%%%%%%%%%%%%%%%%%%%%%%%%%%%%%%%%%%%%%%%%%%%%%%%%%%%%%%%%%%%%%%%%

\section{Inclusions of dihedral groups in symmetric groups}

\begin{lemma}\label{lemma:cyc}
Let $n\geq 4$, let $G:=\mathfrak{S}_n$, and let
$H:=\langle a\rangle$ where $a$ is an
$n$-cycle in $G$. Then the ring extension $\mathbb{C}H\subset\mathbb{C}G$ has
depth $3$.
\end{lemma}

\begin{proof}
We may suppose that $a=(1,2,\ldots,n)$.
Since $H$ is not normal in $G$, $\mathbb{C}H\subset\mathbb{C}G$ is not of
depth 2. We set $g:=(2,n-2)(n-1,n)$ and
$\tilde{a}:=a^g=(1,n-2,3,4,\ldots, n-3,2,n,n-1)$.
In the case where $n=4$, this means $g=(3,4)$ and thus $\tilde{a}=(1,2,4,3)$.
It suffices to show that $\langle a\rangle\cap\langle\tilde{a}\rangle=1$.
For $n=4$ this is obviously true. Let now $n>4$.
Assume, for a contradiction, that $\langle a\rangle\cap\langle\tilde{a}\rangle\neq 1$.
Then there are some $1\neq l\in\mathbb{N}$, some $k\in \mathbb{N}$, and some
$i\in\{1,\ldots, l-1\}$ such that $n=kl$ and $a^{ki}=\tilde{a}^k\neq 1$.
If $k\in\{2,\ldots,n-4\}$ then $1+ki=a^{ki}(1)=\tilde{a}^k(1)=1+k$. Thus
$i=1$, and if $k<n-4$ then we have the contradiction
$2+k=\tilde{a}^k(n-2)=a^k(n-2)\equiv n-2+k\pmod{n}$.
If $k=n-4$ then $2=\tilde{a}^k(n-2)=a^k(n-2)\equiv n-2+k\pmod{n}$, hence
$k=4$ and $n= 8$. But then $a^4(n)=4\neq 3=\tilde{a}^4(n)$, a
contradiction. If $k=1$ then $1+i=a^i(1)=\tilde{a}(1)=n-2$, thus
$i=n-3$. But $a^{n-3}(2)=n-1\neq n=\tilde{a}(2)$, a contradiction. Finally,
let $k\in\{n-1,n-2,n-3\}$. Since $n=kl\geq 2k\geq 2(n-3)=2n-6$, we get
$n\leq 6$. Thus $n=6$, $k=3$, $l=2$, $i=1$, and we have a contradiction.
\end{proof}

\begin{prop}\label{prop:dihedral}
Let $n>5$, and let $H:=D_{2n}$ be the dihedral subgroup
of $\mathfrak{S}_n$ of order $2n$, with generators
$a:=(1,2,\ldots,n)$ and $b:=(1,n)(2,n-1)(3,n-2)\cdots
(\lfloor \frac{n}{2}\rfloor,\lceil \frac{n+2}{2}\rceil)$.
Then the ring extension $\mathbb{C}H\subset \mathbb{C}\mathfrak{S}_n$
has depth $3$.
\end{prop}

\begin{proof}
Since $H$ is not normal in $\mathfrak{S}_n$, the extension
$\mathbb{C}H\subset \mathbb{C}\mathfrak{S}_n$ is not of depth 2. Again we set
$g:=(2,n-2)(n-1,n)$, $\tilde{a}:=a^g=(1,n-2,3,\ldots, n-3,2,n,n-1)$, and
$\tilde{b}:=b^g=(1,n-1)(n-2,n)(2,3)(4,n-3)\cdots (\lfloor \frac{n}{2}\rfloor,\lceil \frac{n+2}{2}\rceil)$.
Also here it suffices to show that $H\cap H^g=1$.
For $n=6$ this is obviously true. Thus, for the remainder of the proof, let $n\geq 7$.
Assume, for
a contradiction, that there is some $1\neq x\in H\cap H^g$.
Note that $H=\{a^i,\, a^ib\mid i=0,\ldots,n-1\}$ and
$H^g=\{\tilde{a}^i,\, \tilde{a}^i\tilde{b}\mid i=0,\ldots,n-1\}$. We
distinguish between four cases.\\

\textit{Case 1.} $x=a^i=\tilde{a}^j$, for some $i,j\in\{1,\ldots,n-1\}$. Then
the proof of Lemma \ref{lemma:cyc} leads to a contradiction.\\

\textit{Case 2.} $x=a^i=\tilde{a}^j\tilde{b}$, for some $i,j\in\{0,\ldots,n-1\}$.
That is $a^{2i}=\tilde{a}^j\tilde{b}\tilde{a}^j\tilde{b}=\tilde{a}^j\tilde{a}^{-j}=1$,
hence $n$ is even, and $i=n/2$. In particular,
$$5\leq 1+\frac{n}{2}=a^i(1)=\tilde{a}^j(\tilde{b}(1))=\tilde{a}^j(n-1)$$
which implies $j=1+n/2$. So $4+n/2=a^i(4)=\tilde{a}^j(\tilde{b}(4))=\tilde{a}^j(n-3)$.
But, on the other hand, $\tilde{a}^j(8-3)=6$,  $\tilde{a}^j(10-3)=3$, and
$\tilde{a}^j(n-3)=(n-4)/2$ for $n>10$. In any case this
is not equal to $4+n/2$, a contradiction.\\

\textit{Case 3.} $x=a^ib=\tilde{a}^j$, for some $i,j\in\{0,\ldots,n-1\}$.
But then $x^g=a^j=\tilde{a}^i\tilde{b}$ which is impossible, by the
considerations in case 2 above.\\

\textit{Case 4.} $x=a^ib=\tilde{a}^j\tilde{b}$, for some $i,j\in\{0,\ldots,n-1\}$.
Therefore, $1+i=a^i(1)=a^i(b(n))=\tilde{a}^j(\tilde{b}(n))=\tilde{a}^j(n-2)$.
Suppose that $1\leq j\leq n-5$ so that $1+i=\tilde{a}^j(n-2)=a^j(2)=2+j$.
That is $i=j+1\in\{2,\ldots,n-4\}$.
However, if $i=j+1=2$ then $\tilde{a}^j(\tilde{b}(1))=1\neq 2=a^i(b(1))$,
if $i=j+1=3$ then $\tilde{a}^j(\tilde{b}(1))=n-2\neq 3=a^i(b(1))$, and if $3<i=j+1\leq n-4$ then
$\tilde{a}^j(\tilde{b}(1))=j\neq j+1=i=a^i(b(1))$. Moreover, if $j=0$ then
we have $1+i=n-2$, thus $i=n-3$. But $\tilde{a}^0(\tilde{b}(1))=n-1\neq n-3=a^{n-3}(b(1))$.
Consequently, we must have $j\in\{n-4,n-3,n-2,n-1\}$.  \\

Suppose $j=n-1$ so that $a^ib=\tilde{a}^{-1}\tilde{b}$. Then
$1+i=a^i(1)=a^i(b(n))=\tilde{a}^{-1}(\tilde{b}(n))=\tilde{a}^{-1}(n-2)=1$, thus $i=0$ and
$b=\tilde{a}^{-1}\tilde{b}$. But $b(n-1)=2\neq n-1=\tilde{a}^{-1}(\tilde{b}(n-1))$, a contradiction.\\

Next suppose that $j=n-2$ where $a^ib=\tilde{a}^{-2}\tilde{b}$, and
$1+i=a^i(b(n))=\tilde{a}^{-2}(\tilde{b}(n))=n-1$. Hence $i=n-2$ which gives the
contradiction $a^{-2}(b(2))=a^{-2}(n-1)=n-3\neq 1=\tilde{a}^{-2}(3)=\tilde{a}^{-2}(\tilde{b}(2))$.\\

If $j=n-3$ then $a^ib=\tilde{a}^{-3}\tilde{b}$, and so $i=n-1$. But this time we get
the contradiction $1=\tilde{a}^{-3}(4)=\tilde{a}^{-3}(\tilde{b}(n-3))=a^{-1}(b(n-3))=a^{-1}(4)=3$.\\

Lastly, assume that $j=n-4$ so that $a^ib=\tilde{a}^{-4}\tilde{b}$, and $i=1$.
But, for $n\neq 8$, this implies
$n-6=\tilde{a}^{-4}(2)=\tilde{a}^{-4}(\tilde{b}(3))=a(b(3))=a(n-2)=n-1$, and for $n=8$
we get $6=\tilde{a}^{-4}(2)=\tilde{a}^{-4}(\tilde{b}(3))=a(b(3))=a(6)=7$.
Hence we have again a contradiction.\\

To summarize, neither of the four cases above can occur, and
the assertion of the proposition follows.
\end{proof}

\begin{remark}\label{rem:dihedral}\normalfont
(a) By Example 2.6, we know that the ring extension
$\mathbb{C}D_8\subset\mathbb{C}\mathfrak{S}_4$
has depth 4.\\

(b) The
inclusion matrix of the groups $D_{10}<\mathfrak{S}_5$ is
$$M=\begin{pmatrix}1 & 0& 1& 0&1& 0& 1\\ 0&0&0&2&0&0&0\\ 0&1&1&1&1&1&0\\ 0&1&1&1&1&1&0
\end{pmatrix},$$
which has depth 5. Hence the same applies to the ring extension
$\mathbb{C}D_{10}\subset\mathbb{C}\mathfrak{S}_5$.
\end{remark}

\end{appendix}

\end{document}